\documentclass[a4paper]{amsart}
\usepackage{latexsym,bm,stmaryrd}
\usepackage{amsmath,amsthm,amsfonts,amssymb,mathrsfs,pb-diagram}
\usepackage{tikz}
\usetikzlibrary{arrows,decorations.markings}

\advance\textheight by 1.5mm
\usepackage{cite}
\usepackage{graphics}
\usepackage[all]{xy}

\usepackage[enableskew,vcentermath]{youngtab}
\def\tab(#1){\mbox{\small$\young(#1)$}\,}

\let\<=\langle
\let\>=\rangle
\def\({\big(}
\def\){\big)}

\let\Sym=\BS

\newcommand{\N}{\mathbb N}
\newcommand{\Z}{\mathbb Z}

\DeclareMathAlphabet{\mathpzc}{OT1}{pzc}{m}{it}

\newcommand{\Q}{\mathbb Q}

\def\tlam{{\mathfrak t}^\lam}

\def\tllam{{\mathfrak t}_\lam}

\def\s{{\mathfrak s}}
\def\t{{\mathfrak t}}
\def\u{{\mathfrak u}}
\def\v{{\mathfrak v}}

\newcommand{\lam}{\lambda}

\newcommand\blam{{\boldsymbol\lambda}}

\DeclareMathOperator\Shape{Shape}
\DeclareMathOperator\Std{Std}
\def\SStd(#1,#2){\Std^\Lambda_{#2}(#1)}

\renewcommand\t{\mathfrak{t}}
\renewcommand\u{\mathfrak{u}}

\newcounter{main}
\theoremstyle{plain}

\swapnumbers \numberwithin{equation}{section}
\newtheorem{prop}[equation]{Proposition}
\newtheorem{thm}[equation]{Theorem}
\newtheorem{cor}[equation]{Corollary}
\newtheorem{lem}[equation]{Lemma}
\newtheorem{conj}[equation]{Lusztig's Conjecture}
\theoremstyle{definition}
\newtheorem{dfn}[equation]{Definition}
\theoremstyle{remark}
\newtheorem{rem}[equation]{Remark}

  {\refstepcounter{equation}\trivlist
   \item[\hskip\labelsep\theequation.~\textbf{Example#1}\space]
   \ignorespaces
  }{\unskip\nobreak\hfil%
    \penalty50\hskip2em\hbox{}\nobreak\hfil$\Diamond$%
    \parfillskip=0pt\finalhyphendemerits=0\penalty-100\endtrivlist}

{\catcode`\|=\active
  \gdef\set#1{\mathinner{\lbrace\,{\mathcode`\|"8000%
                                   \let|\midvert #1}\,\rbrace}}
}
\def\midvert{\egroup\mid\bgroup}

\begin{document}
\title
{On involutions in symmetric groups and a conjecture of Lusztig}


  \author{Jun Hu}
  \address{School of Mathematics and Statistics\\
  Beijing Institute of Technology\\
  Beijing, 100081, P.R. China}
  \email{junhu404@bit.edu.cn}

  \author{Jing Zhang}
   \address{School of Mathematics and Statistics\\
  Beijing Institute of Technology\\
  Beijing, 100081, P.R. China}
  \email{ellenbox@bit.edu.cn}

\bibliographystyle{andrew}

\begin{abstract} Let $(W, S)$ be a Coxeter system equipped with a fixed automorphism $\ast$ of order $\leq 2$ which preserves $S$. Lusztig (and with Vogan in some special cases)
have shown that the space spanned by set of ``twisted" involutions (i.e., elements $w\in W$ with $w^\ast=w^{-1}$) was naturally endowed with a module structure of the Hecke algebra of
$(W, S)$ with two distinguished bases, which can be viewed as twisted analogues of the well-known standard basis and Kazhdan-Lusztig basis. The transition matrix between these
bases defines a family of polynomials $P_{y,w}^{\sigma}$ which can be viewed as ``twisted" analogues of the well-known Kazhdan-Lusztig polynomials of $(W, S)$. Lusztig has conjectured that this module is isomorphic to the right ideal of the Hecke algebra (with Hecke parameter $u^2$) associated to $(W,S)$ generated by the element $X_{\emptyset}:=\sum_{w^\ast=w}u^{-\ell(w)}T_w$. In this paper we prove this conjecture in the case when $\ast=\text{id}$ and $W=\Sym_n$ (the symmetric group on $n$ letters). Our methods are expected to be generalised to all the other finite crystallographic Coxeter groups.
\end{abstract}

\maketitle

\section{Introduction}

Let $(W,S)$ be a fixed Coxeter system with length function $\ell: W\rightarrow\mathbb{N}$. If $w\in W$ then by definition
$$\ell(w):=\min\{k|w=s_{i_1}\dots s_{i_k} \text{ for some }s_{i_1},\dots,s_{i_k}\in S\}.$$
Let ``$\leq$" be the Bruhat partial ordering on $W$.
Let ``$\ast$" be a fixed automorphism of $W$ with order $\leq 2$ and such that $s^{\ast}\in S$ for any $s\in S$.

\begin{dfn} \label{twistedinvolutions} We define $$
I_{\ast}:=\bigl\{w\in W\bigm|w^{\ast}=w^{-1}\bigr\}.
$$
The elements of $I_{\ast}$ will be called twisted involutions.
\end{dfn}

If $\ast=\text{id}_W$ (the identity automorphism on $W$), then the elements of $I_{\ast}$ will be called involutions.

Let $v$ be an indeterminate over $\Z$ and $u:=v^2$. Set $\mathcal{A}:=\Z[u,u^{-1}]$. Let $\mathcal{H}_u$ be the Iwahori-Hecke algebra associated to $(W,S)$ with Hecke parameter $u^2$ and defined over $\mathcal{A}$. By definition, $\mathcal{H}_u$ is a free $\mathcal{A}$-module with basis $\{T_w\}_{w\in W}$. There is a unital $\mathcal{A}$-algebra structure on $\mathcal{H}_u$ with unit $T_1$ and such that $$\begin{aligned}
& T_w T_{w'}=T_{ww'}\quad\text{if $\ell(ww')=\ell(w)+\ell(w')$; and}\\
& (T_s+1)(T_s-u^2)=0\quad\text{for all $s\in S$.}
\end{aligned}
$$

Let $M$ be the free $\mathcal{A}$-module with basis $\{a_{w}|w\in I_{\ast}\}$. The following result was obtained by Lusztig and Vogan (\cite{LV1}) in the special case where $W$ is a Weyl group or an affine Weyl group, and by Lusztig (\cite{Lu1}) in the general case.

\begin{thm} \text{(\cite{LV1},\,\,\cite[Theorem 0.1]{Lu1})} \label{LVaction} There is a unique $\mathcal{H}_u$-module structure on $M$ such that for any $s\in S$ and any $w\in I_\ast$ we have that $$\begin{aligned}
& T_s a_w=ua_w+(u+1)a_{sw}\quad\text{if\, $sw=ws^{\ast}>w$;}\\
& T_s a_w=(u^2-u-1)a_w+(u^2-u)a_{sw}\quad\text{if\, $sw=ws^{\ast}<w$;}\\
& T_s a_w=a_{sws^{\ast}}\quad\text{if\, $sw\neq ws^{\ast}>w$;}\\
& T_s a_w=(u^2-1)a_w+u^2a_{sws^{\ast}}\quad\text{if\, $sw\neq ws^{\ast}<w$.}\\
\end{aligned}
$$
\end{thm}

Set $\underline{\mathcal{A}}:=\Z[v,v^{-1}]$. Then $\mathcal{A}$ can be naturally regarded as a subring of $\underline{\mathcal{A}}$ because $u=v^2$. Let
$\underline{\mathcal{H}}:=\underline{\mathcal{A}}\otimes_{\mathcal{A}}\mathcal{H}_u$.
Let $-: \underline{\mathcal{A}}\rightarrow\underline{\mathcal{A}}$ be the ring involution such that $\overline{v^n}=v^{-n}$ for $n\in\Z$. We denote by $-: \underline{\mathcal{H}}\rightarrow\underline{\mathcal{H}}$
the ring involution such that $\overline{v^nT_x}=v^{-n}T_{x^{-1}}^{-1}$ for any $x\in W, n\in\N$. Then ``$-$" restricts to the unique ring involution of $\mathcal{H}_u$ such that $\overline{u^nT_x}=u^{-n}T_{x^{-1}}^{-1}$ for any $x\in W, n\in\N$ (cf.\,\, \cite{KL}).

In\,\,\cite{LV1} and\,\,\cite[Theorem 0.2]{Lu1}, Lusztig and Vogan have shown that there exists a unique $\Z$-linear map $-: M\rightarrow M$ such that $\overline{hm}=\overline{h}\overline{m}$ for all $h\in\mathcal{H}_u$, $m\in M$ and $\overline{a_1}=a_1$. For any $m\in M$, $\overline{\overline{m}}=m$. Moreover, for any $w\in I_\ast$, $\overline{a_w}=(-1)^{\ell(w)}T_{w^{-1}}^{-1}a_{w^{-1}}$,

Set $\underline{M}:=\underline{\mathcal{A}}\otimes_{\mathcal{A}}M$. The map ``$-: M\rightarrow M$" can be naturally extended to a $\Z$-linear map $-:\underline{M}\rightarrow\underline{M}$ such that $\overline{v^nm}=v^{-n}\overline{m}$
for $m\in \underline{M}$, $n\in\Z$. For each $w\in I_\ast$, Lusztig and Vogan have proved further that there is a unique element $$
A_w=v^{-\ell(w)}\sum_{y\in I_\ast, y\leq w}P_{y,w}^{\sigma}a_y\in \underline{M},
$$
where $P_{y,w}^{\sigma}\in\Z[u]$ such that $\overline{A_w}=A_w$, $P_{w,w}^{\sigma}=1$ and for any $y\in I_\ast$, $y<w$, we have $\deg P_{y,w}^{\sigma}\leq(\ell(w)-\ell(y)-1)/2$. Furthermore, the elements $\{A_w|w\in I_\ast\}$ form an $\underline{\mathcal{A}}$-basis of $\underline{M}$. The polynomials $P_{w,w}^{\sigma}$ can be viewed as a ¡°twisted¡± analogue of the well-known Kazhdan-Lusztig polynomial $P_{y,w}$ of $(W, S)$ (\cite{KL}), and the $\underline{\mathcal{A}}$-basis $\{A_w|w\in I_\ast\}$ can be viewed  as a ¡°twisted¡± analogue of the well-known Kazhdan-Lusztig basis $C'_w$ (\cite[1.1.c]{KL}).

\begin{dfn} \text{(\cite{Lu2})} Define $$
X_{\emptyset}:=\sum_{x\in W, x^\ast=x}u^{-\ell(x)}T_x .
$$
\end{dfn}

Let $\Q(u)$ be the field of rational functions on $u$. Set $\mathcal{H}^{\Q(u)}:=\Q(u)\otimes_{\mathcal{A}}\mathcal{H}_u$. In\,\, \cite[3.4(a)]{Lu2}, Lusztig proposed the following conjecture:

\begin{conj} \text{(\cite[3.4(a)]{Lu2})} \label{LC} With the notations as above, there is a unique isomorphism of $\mathcal{H}^{\Q(u)}$-modules
$\eta: \Q(u)\otimes_{\mathcal{A}}M\cong \mathcal{H}^{\Q(u)}X_{\emptyset}$ such that $a_1\mapsto X_{\emptyset}$.
\end{conj}

The purpose of this paper is to give a proof of this conjecture in the case when $\ast=\text{id}_W$ and $W$ is the symmetric group $\Sym_n$ on $n$ letters (i.e., the Weyl group of type $A_{n-1}$) for any $n\in\N$. Our methods are expected to be generalised to all the other finite crystallographic Coxeter groups. The case when $\ast=\text{id}_W$ and $W$ is the Weyl group of types $D_{n}$ and $B_n$ will be dealt with in forthcoming papers. As a byproduct of this paper, we show that any two reduced $I_\ast$-expressions for an involution in $\Sym_n$ can be transformed into each other through a series of braid $I_\ast$-transformations, which can be viewed as a ``twisted" analogue of a  well-known classical fact of Matsumoto (\cite{Mathas}) which said that any two reduced expressions for an element in $\Sym_n$ can be transformed into each other through a series of braid transformations.

The paper is organised as follows. In Section 2, we first recall some preliminary and known results (due to Hultman) on reduced $I_{\ast}$-expressions for twisted involutions, then we introduce a new notion of braid $I_\ast$-transformations and show in Lemma \ref{braid0} that any braid $I_\ast$-transformations on reduced $I_\ast$-sequence for a given involution in $\Sym_n$ do not change the involution itself. We also give a number of technical lemmas which will be used in the next section. In Section 3, we prove in Theorem \ref{mainthm0} that any two reduced $I_\ast$-expressions for an involution in $\Sym_n$ can be transformed into each other through a series of braid $I_\ast$-transformations. This key result will play a central role in the proof of the main result Theorem \ref{mainthm2}. In Section 4, we use the Young seminormal bases theory for the semisimple Iwahori-Hecke algebra of type $A_{n-1}$ to show that the dimension of $\mathcal{H}^{\Q(u)}X_{\emptyset}$ is bigger or equal than the number of involutions in $W$. The main result of this paper is given in Section 5, where we prove Lusztig's Conjecture \ref{LC} in the case when $\ast=\text{id}_W$ and $W$ is symmetric group $\Sym_n$ for any $n\in\N$.
\medskip

\section*{Acknowledgements}

Both authors were supported by the National Natural Science Foundation of China (NSFC 11171021, NSFC 11471315). The first author also thanks Professor Hultman for some helpful discussions.

\bigskip
\section{reduced $I_{\ast}$-expressions}

In this section we shall give some preliminary and known results on reduced $I_{\ast}$-expressions for twisted inviolutions.

\begin{dfn} \label{twistedinvolutions2} For any $w\in I_\ast$ and $s\in S$, we define $$
s\ltimes w:=\begin{cases} sw &\text{if $sw=ws^\ast$;}\\
sws^\ast &\text{if $sw\neq ws^\ast$.}
\end{cases}
$$
For any $w\in I_\ast$ and $s_{i_1},\cdots,s_{i_k}\in S$, we define $$
s_{i_1}\ltimes s_{i_2}\ltimes\cdots\ltimes s_{i_k}\ltimes w:=s_{i_1}\ltimes\bigl(s_{i_2}\ltimes\cdots\ltimes (s_{i_k}\ltimes w)\cdots\bigr) .
$$
\end{dfn}

\begin{lem} \label{square} For any $w\in I_\ast$ and $s\in S$, we have that $$
s\ltimes(s\ltimes w)=w .
$$
\end{lem}

\begin{proof} If $sw=ws^\ast$, then $s\ltimes w=sw$. In this case, $s(sw)=w=(sw)s^\ast$, so by definition, $$
s\ltimes(s\ltimes w)=s\ltimes(sw)=w .
$$

If $sw\neq ws^\ast$, then $s\ltimes w=sws^\ast$. Now $s(sws^\ast)=ws^\ast\neq sw=(sws^\ast)s^\ast$, so by definition, $$
s\ltimes(s\ltimes w)=s\ltimes(sws^\ast)=s(sws^\ast)s^\ast=w . $$
This completes the proof of the lemma. \end{proof}

\begin{rem} \label{rem1} In general, the operation $\ltimes: S\times I_\ast\rightarrow I_\ast$ does not extend to a group action of $W$ on $I_\ast$. For example, let $W=\Sym_4$ (the symmetric group on $\{1,2,3,4\}$), $w=s_2=(2,3)$, then
$$
s_1\ltimes(s_2\ltimes (s_1\ltimes w))=1\neq s_2s_1s_2=s_2\ltimes(s_1\ltimes(s_2\ltimes w)) .
$$
\end{rem}

It is well-known that every element $w\in I_\ast$ is of the form $w=s_{i_1}\ltimes s_{i_2}\ltimes\cdots\ltimes s_{i_k}$ for some $k\in\N$ and $s_{i_1},\cdots,s_{i_k}\in S$.

As we shall see later in this paper, the main obstacle for the failure of a group action in the example given in Remark \ref{rem1} is the fact that $s_1\ltimes(s_2\ltimes (s_1\ltimes w))$ is not a reduced $I_{\ast}$-expression in the sense of the following definition.

\begin{dfn} \label{Ireduced} (\!\cite{Hu1}, \cite{Hu2}) Let $w\in I_\ast$. If $w=s_{i_1}\ltimes s_{i_2}\ltimes\cdots\ltimes s_{i_k}$, where $k\in\N$, $s_{i_j}\in S$ for each $j$, then $(s_{i_1},\cdots,s_{i_k})$ is called an $I_\ast$-expression for $w\in I_\ast$. Such an $I_\ast$-expression for $w\in I_\ast$ is reduced if its length $k$ is minimal.
\end{dfn}

We regard the empty sequence $()$ as a reduced $I_\ast$-expression for $w=1$. If follows by induction on $\ell(w)$ that every element of $w\in I_\ast$ has a reduced $I_\ast$-expression.

\begin{lem}\label{rankfunc}(\cite{Hu1}, \cite{Hu2}) Let $w\in I_\ast$. Any reduced $I_\ast$-expression for $w$ has a common length. Let $\rho: I_\ast\rightarrow\N$ be the map which assign $w\in I_\ast$ to this common length. Then $(I_\ast,\leq)$ is a graded poset with rank function $\rho$. Moreover, if $s\in S$ then $\rho(s\ltimes w)=\rho(w)\pm 1$, and $\rho(s\ltimes w)=\rho(w)-1$ if and only if $\ell(sw)=\ell(w)-1$.
\end{lem}

\begin{cor} \label{length0} Let $w\in I_\ast$ and $s\in S$. Suppose that $sw\neq ws^\ast$. Then $\ell(sw)=\ell(w)+1$ if and only if $\ell(ws^\ast)=\ell(w)+1$, and if and only if $\ell(s\ltimes w)=\ell(w)+2$. The same is true if we replace ``$+$" by ``$-$".
\end{cor}

\begin{proof} This follows from Lemma \ref{rankfunc}.
\end{proof}

\begin{cor} \label{deletion2} Let $w\in I_\ast$ and $s\in S$. Suppose that $\rho(w)=k$. If $sw<w$ then $w$ has a reduced $I_\ast$-expression which is of the form $s\ltimes s_{j_1}\ltimes\cdots\ltimes s_{j_{k-1}}$.
\end{cor}

\begin{proof} This follows from Lemma \ref{rankfunc} and the fact (Lemma \ref{square}) that $w=s\ltimes(s\ltimes w)$.
\end{proof}

\begin{dfn} \label{ReducedSequence} Let $w\in I_\ast$ and $s_{i_1},\cdots,s_{i_k}\in S$. If $$
\rho(s_{i_1}\ltimes s_{i_2}\ltimes\cdots\ltimes s_{i_k}\ltimes w)=\rho(w)+k ,
$$
then we shall call the sequence $(s_{i_1},\cdots,s_{i_k},w)$ reduced, or $(s_{i_1},\cdots,s_{i_k},w)$  a reduced sequence.
\end{dfn}

In particular, any reduced $I_\ast$-expression for $w\in I_\ast$ is automatically a reduced sequence.
In the sequel, by some abuse of notations, we shall also call $(i_1,\cdots,i_k)$ a reduced sequence whenever $(s_{i_1},\cdots,s_{i_k})$ is a reduced sequence in the sense of Definition \ref{ReducedSequence}.

\begin{rem} Let $s_{i_1},\cdots,s_{i_k}\in S$ and $1\leq a\leq k$. We shall use the expression \begin{equation}\label{omit}
s_{i_1}\ltimes\cdots\ltimes s_{i_{a-1}}\ltimes s_{i_{a+1}}\ltimes\cdots\ltimes s_{i_k}
\end{equation}
to denote the element obtained from omitting ``$s_{i_a}\ltimes$" in the expression $s_{i_1}\ltimes\cdots\ltimes s_{i_k}$. In particular, if $a=1$ then (\ref{omit}) denotes the element $s_{i_2}\ltimes\cdots\ltimes s_{i_k}$; while if $a=k$ then (\ref{omit}) denotes the element $s_{i_1}\ltimes\cdots\ltimes s_{i_{k-1}}$.
This convention will be adopted throughout this paper.
\end{rem}

\begin{prop} \text{(Exchange Property,\,\, \cite[Prop. 3.10]{Hu2})} \label{exchange} Suppose $(s_{i_1}, \cdots,s_{i_k})$ is a reduced $I_\ast$-expression for $w\in I_\ast$ and that $\rho(s\ltimes s_{i_1}\ltimes s_{i_2}\ltimes\cdots\ltimes s_{i_k})<k$ for some $s\in S$. Then$$
s\ltimes s_{i_1}\ltimes s_{i_2}\ltimes\cdots\ltimes s_{i_k}=s_{i_1}\ltimes s_{i_2}\ltimes\cdots\ltimes{s_{i_{a-1}}}\ltimes{s_{i_{a+1}}}\ltimes\cdots\ltimes s_{i_k}
$$
for some $a\in\{1,2,\cdots,k\}$.
\end{prop}

\bigskip

\textbf{From now on and until the end of this paper, we assume that $W=\Sym_n$, the symmetric group on $n$ letters, where $n\in\N$.} Moreover, we assume that ``$\ast=\text{id}$" is the identity map on $\Sym_n$. In particular, $$
I_{\ast}=\{w\in\Sym_n|w^2=1\}
$$
is the set of involutions in $\Sym_n$. For each $1\leq i<n$, we define $$
s_i:=(i,i+1).
$$

In this case, if $w=1$ (the identity element of $\Sym_n$), then by definition for any $s\in S$, $$
s\ltimes w=s\ltimes 1=s .
$$

\begin{dfn}\label{braid1} Let $w\in I_\ast$. By a braid $I_\ast$-transformation, we mean one of the following transformations: $$
\begin{aligned}
&(s_{i_1},\cdots,s_{i_a},s_j,s_{j+1},s_j,s_{l_1},\cdots,s_{l_t},w)\longmapsto\\
&\qquad\qquad (s_{i_1},\cdots,s_{i_a},s_{j+1},s_j,s_{j+1},s_{l_1},\cdots,s_{l_t},w) ,\\
&(s_{i_1},\cdots,s_{i_a},s_{j+1},s_{j},s_{j+1},s_{l_1},\cdots,s_{l_t},w)\longmapsto\\
&\qquad\qquad (s_{i_1},\cdots,s_{i_a},s_{j},s_{j+1},s_{j},s_{l_1},\cdots,s_{l_t},w) ,\\
& (s_{i_1},s_{i_2},\cdots,s_{i_a},s_b,s_{c},s_{l_1},\cdots,s_{l_t},w')\longmapsto \\
&\qquad\qquad (s_{i_1},s_{i_2},\cdots,s_{i_a},s_{c},s_b,s_{l_1},\cdots,s_{l_t},w') ,\\
& (s_{i_1},s_{i_2},\cdots,s_{i_a},s_k,s_{k+1})\longmapsto
(s_{i_1},s_{i_2},\cdots,s_{i_a},s_{k+1},s_k) ,\\
& (s_{i_1},s_{i_2},\cdots,s_{i_a},s_{k+1},s_{k})\longmapsto
(s_{i_1},s_{i_2},\cdots,s_{i_a},s_{k},s_{k+1}) ,
\end{aligned}
$$
where $w,w'\in I_\ast$, $1\leq i_1,\cdots,i_a, l_1,\cdots,l_t, b, c<n$, $1\leq j,k<n-1$, $|b-c|>1$, and the sequences appeared above are all reduced sequences.\footnote{Note that our assumption that these sequences are all reduced implies that $w\neq 1$ whenever $t=0$.}
\end{dfn}

Let $w\in I_\ast$ and $s_{i_1},\cdots,s_{i_k}\in S$. By definition, it is clear that $(s_{i_1},\cdots,s_{i_k},w)$ is a reduced sequence if and only if $(s_{i_1},\cdots,s_{i_k},s_{j_1},\cdots,s_{j_t})$ is a reduced sequence for some (and any) reduced $I_\ast$-expression $(s_{j_1},\cdots,s_{j_t})$ of $w$.

\begin{dfn} \label{braid3} Let $(s_{i_1},\cdots,s_{i_k},w), (s_{j_1},\cdots,s_{j_l},u)$ be two reduced $I_\ast$-sequences, where $w,u\in I_{\ast}$. We shall write $(s_{i_1},\cdots,s_{i_k},w)\longleftrightarrow (s_{j_1},\cdots,s_{j_l},u)$ whenever there exists a series braid $I_\ast$-transformations which transform $$(s_{i_1},\cdots,s_{i_k},s_{l_1},\cdots,s_{l_b})$$ into $(s_{j_1},\cdots,s_{j_l},s_{p_1},\cdots,\cdots,s_{p_c})$, where $(s_{l_1},\cdots,s_{l_b})$ and $(s_{p_1},\cdots,s_{p_c})$ are some reduced $I_\ast$-expressions of $w$ and $u$ respectively. Moreover, we shall also write $$
(i_1,\cdots,i_k)\longleftrightarrow (j_1,\cdots,j_k)
$$
whenever $(s_{i_1},\cdots,s_{i_k})\longleftrightarrow (s_{j_1},\cdots,s_{j_k})$.
\end{dfn}

\begin{lem} \label{braid0} Let $w\in I_\ast$. Let $(s_{i_1},\cdots,s_{i_k},w)$ be a reduced sequence.

1) If $|i_{k-1}-i_k|>1$, then $(s_{i_1},\cdots,s_{i_{k-2}}, s_{i_k}, s_{i_{k-1}}, w)$ is a reduced sequence too, and
$$
s_{i_1}\ltimes s_{i_2}\ltimes\cdots\ltimes s_{i_{k-1}}\ltimes s_{i_k}\ltimes w=s_{i_1}\ltimes s_{i_2}\ltimes\cdots\ltimes s_{i_{k-2}}\ltimes s_{i_k}\ltimes s_{i_{k-1}}\ltimes w .
$$

2) If $i_{k-2}=i_{k}=i_{k-1}\pm 1$, then $(s_{i_1},s_{i_2},\cdots,s_{i_{k-3}}, s_{i_{k-1}}, s_{i_k}, s_{i_{k-1}}, w)$ is a reduced sequence too, and
$$\begin{aligned}
& s_{i_1}\ltimes s_{i_2}\ltimes\cdots\ltimes s_{i_{k-3}}\ltimes s_{i_{k-2}}\ltimes s_{i_{k-1}}\ltimes s_{i_k}\ltimes w\\
&\qquad =s_{i_1}\ltimes s_{i_2}\ltimes\cdots\ltimes s_{i_{k-3}}\ltimes s_{i_{k-1}}\ltimes s_{i_k}\ltimes s_{i_{k-1}}\ltimes w .
\end{aligned}$$

3) If $w=s_{i_{k}\pm 1}$, then $(s_{i_1},s_{i_2},\cdots,s_{i_{k-1}}, w, s_{i_{k}})$ is a reduced sequence too, and
$$
s_{i_1}\ltimes s_{i_2}\ltimes\cdots\ltimes s_{i_{k-1}}\ltimes s_{i_k}\ltimes w=s_{i_1}\ltimes s_{i_2}\ltimes\cdots\ltimes s_{i_{k-1}}\ltimes w\ltimes s_{i_{k}} .
$$
\end{lem}

\begin{proof} 1) This follows from the fact that $s_{i_{k-1}}s_{i_k}=s_{i_k}s_{i_{k-1}}$ and some direct case by case check, see also \cite[Lemma 3.24]{Vor}.\smallskip

2) It suffices to show that
$$
s_{i_k}\ltimes s_{i_{k-1}}\ltimes s_{i_k}\ltimes w=s_{i_{k-1}}\ltimes s_{i_k}\ltimes s_{i_{k-1}}\ltimes w .
$$
There are eight possibilities:

\smallskip
{\it Case 1.} $s_{i_k}w\neq ws_{i_k}$, $s_{i_{k-1}}s_{i_k}ws_{i_k}\neq s_{i_k}ws_{i_k}s_{i_{k-1}}$ and $
s_{i_k}s_{i_{k-1}}s_{i_k}ws_{i_k}s_{i_{k-1}}\neq s_{i_{k-1}}s_{i_k}ws_{i_k}s_{i_{k-1}}s_{i_k}$. In this case, we have that
$$
s_{i_k}\ltimes s_{i_{k-1}}\ltimes s_{i_k}\ltimes w=s_{i_k}s_{i_{k-1}}s_{i_k}ws_{i_k}s_{i_{k-1}}s_{i_k}
=s_{i_{k-1}}s_{i_k}s_{i_{k-1}}ws_{i_{k-1}}s_{i_k}s_{i_{k-1}}.
$$
Since $(s_{i_k},s_{i_{k-1}},s_{i_k},w)$ is a reduced sequence, it is clear (by Lemma \ref{rankfunc}) that $$
\ell(s_{i_{k-1}}s_{i_k}s_{i_{k-1}}ws_{i_{k-1}}s_{i_k}s_{i_{k-1}})=\ell(w)+6. $$
So $s_{i_{k-1}}w\neq ws_{i_{k-1}}$, $s_{i_k} s_{i_{k-1}} w s_{i_{k-1}}\neq s_{i_{k-1}} w s_{i_{k-1}} s_{i_k}$ and $
s_{i_{k-1}}s_{i_k}s_{i_{k-1}}ws_{i_{k-1}}s_{i_k}\neq s_{i_k}s_{i_{k-1}}ws_{i_{k-1}}s_{i_k}s_{i_{k-1}}$. By definition, we get that
$$
s_{i_{k-1}}\ltimes s_{i_k}\ltimes s_{i_{k-1}}\ltimes w=s_{i_{k-1}}s_{i_k}s_{i_{k-1}}ws_{i_{k-1}}s_{i_k}s_{i_{k-1}}
=s_{i_k}\ltimes s_{i_{k-1}}\ltimes s_{i_k}\ltimes w.
$$

\smallskip
{\it Case 2.} $s_{i_k}w\neq ws_{i_k}$, $s_{i_{k-1}}s_{i_k}ws_{i_k}\neq s_{i_k}ws_{i_k}s_{i_{k-1}}$ and $
s_{i_k}s_{i_{k-1}}s_{i_k}ws_{i_k}s_{i_{k-1}}= s_{i_{k-1}}s_{i_k}ws_{i_k}s_{i_{k-1}}s_{i_k}$. In this case, we have that
$$
s_{i_k}\ltimes s_{i_{k-1}}\ltimes s_{i_k}\ltimes w=s_{i_k}s_{i_{k-1}}s_{i_k}ws_{i_k}s_{i_{k-1}}
=s_{i_{k-1}}s_{i_k}s_{i_{k-1}}ws_{i_k}s_{i_{k-1}}.
$$
Since
$$\begin{aligned}
s_{i_{k-1}}s_{i_k}s_{i_{k-1}}ws_{i_k}s_{i_{k-1}}&=s_{i_k}s_{i_{k-1}}s_{i_k}ws_{i_k}s_{i_{k-1}}
= s_{i_{k-1}}s_{i_k}ws_{i_k}s_{i_{k-1}}s_{i_k}\\
&=s_{i_{k-1}}s_{i_k}ws_{i_{k-1}}s_{i_k}s_{i_{k-1}},
\end{aligned}$$
it follows that $s_{i_{k-1}}w=ws_{i_{k-1}}$. Moreover, since $(s_{i_k},s_{i_{k-1}},s_{i_k},w)$ is a reduced sequence, it follows from Lemma \ref{rankfunc} that $$\ell(s_{i_{k-1}}s_{i_k}s_{i_{k-1}}ws_{i_k}s_{i_{k-1}})=\ell(w)+5 .
$$
In particular, $s_{i_k}s_{i_{k-1}}w\neq s_{i_{k-1}}ws_{i_k}$ and $
s_{i_{k-1}}s_{i_k}s_{i_{k-1}}ws_{i_k}\neq s_{i_k}s_{i_{k-1}}ws_{i_k}s_{i_{k-1}}$. As a consequence, we get (by definition) that
$$
s_{i_{k-1}}\ltimes s_{i_k}\ltimes s_{i_{k-1}}\ltimes w =s_{i_{k-1}}s_{i_k}s_{i_{k-1}}ws_{i_k}s_{i_{k-1}}
=s_{i_k}\ltimes s_{i_{k-1}}\ltimes s_{i_k}\ltimes w.
$$

\smallskip
{\it Case 3.} $s_{i_k}w=ws_{i_k}$, $s_{i_{k-1}}s_{i_k}w\neq s_{i_k}ws_{i_{k-1}}$ and $$
s_{i_k}s_{i_{k-1}}s_{i_k}ws_{i_{k-1}}\neq s_{i_{k-1}}s_{i_k}ws_{i_{k-1}}s_{i_k}.$$ In this case, we have that
$$
s_{i_k}\ltimes s_{i_{k-1}}\ltimes s_{i_k}\ltimes w=s_{i_k}s_{i_{k-1}}s_{i_k}ws_{i_{k-1}}s_{i_k}
=s_{i_{k-1}}s_{i_k}s_{i_{k-1}}ws_{i_{k-1}}s_{i_k}.
$$
Since $(s_{i_k},s_{i_{k-1}},s_{i_k},w)$ is a reduced sequence, it follows from Lemma \ref{rankfunc} that $$
\ell(s_{i_{k-1}}s_{i_k}s_{i_{k-1}}ws_{i_{k-1}}s_{i_k})=\ell(w)+5 .
$$
In particular, $s_{i_{k-1}}w\neq ws_{i_{k-1}}$ and $s_{i_k}s_{i_{k-1}}ws_{i_{k-1}}\neq s_{i_{k-1}}ws_{i_{k-1}}s_{i_k}$. Note that
$$\begin{aligned}
s_{i_{k-1}}s_{i_k}s_{i_{k-1}}ws_{i_{k-1}}s_{i_k}&=s_{i_k}s_{i_{k-1}}s_{i_k}ws_{i_{k-1}}s_{i_k}
= s_{i_k}s_{i_{k-1}}ws_{i_k}s_{i_{k-1}}s_{i_k}\\
&=s_{i_k}s_{i_{k-1}}ws_{i_{k-1}}s_{i_k}s_{i_{k-1}}.
\end{aligned}$$
By definition, we get that
$$
s_{i_{k-1}}\ltimes s_{i_k}\ltimes s_{i_{k-1}}\ltimes w =s_{i_{k-1}}s_{i_k}s_{i_{k-1}}ws_{i_{k-1}}s_{i_k}
=s_{i_k}\ltimes s_{i_{k-1}}\ltimes s_{i_k}\ltimes w.
$$

\smallskip
{\it Case 4.} $s_{i_k}w\neq ws_{i_k}$, $s_{i_{k-1}}s_{i_k}ws_{i_k}=s_{i_k}ws_{i_k}s_{i_{k-1}}$ and $$
s_{i_k}s_{i_{k-1}}s_{i_k} w s_{i_k} \neq s_{i_{k-1}} s_{i_k} w s_{i_k} s_{i_k}.$$ In this case, we have that
$$
s_{i_k}\ltimes s_{i_{k-1}}\ltimes s_{i_k}\ltimes w=s_{i_k}s_{i_{k-1}}s_{i_k}ws_{i_k}s_{i_k}=s_{i_k}s_{i_{k-1}}s_{i_k}w .
$$
Since $(s_{i_k},s_{i_{k-1}},s_{i_k},w)$ is a reduced sequence, it follows from Lemma \ref{rankfunc} that $$\ell(s_{i_k}s_{i_{k-1}}s_{i_k}w)=
\ell(s_{i_k}\ltimes s_{i_{k-1}}\ltimes s_{i_k}\ltimes w)=\ell(w)+5,$$ which is impossible. Therefore, this case can not happen.

\smallskip
{\it Case 5.} $s_{i_k}w\neq ws_{i_k}$, $s_{i_{k-1}}s_{i_k}ws_{i_k}=s_{i_k}ws_{i_k}s_{i_{k-1}}$ and $$
s_{i_k}s_{i_{k-1}}s_{i_k}ws_{i_k}= s_{i_{k-1}}s_{i_k}ws_{i_k}s_{i_k}.$$ In this case, we have that
$$
s_{i_k}\ltimes s_{i_{k-1}}\ltimes s_{i_k}\ltimes w=s_{i_k}s_{i_{k-1}}s_{i_k}ws_{i_k}=s_{i_k}s_{i_{k-1}}ws_{i_k}s_{i_k}=s_{i_k}s_{i_{k-1}}w.
$$
Since $(s_{i_k},s_{i_{k-1}},s_{i_k},w)$ is a reduced sequence, it follows from Lemma \ref{rankfunc} that $$\ell(s_{i_k}s_{i_{k-1}}w)=\ell(s_{i_k}\ltimes s_{i_{k-1}}\ltimes s_{i_k}\ltimes w)=\ell(w)+4,
$$ which is impossible. Therefore, this case can not happen too.

\smallskip
{\it Case 6.} $s_{i_k}w=ws_{i_k}$, $s_{i_{k-1}}s_{i_k}w\neq s_{i_k}ws_{i_{k-1}}$ and $$
s_{i_k}s_{i_{k-1}}s_{i_k}ws_{i_{k-1}}= s_{i_{k-1}}s_{i_k}ws_{i_{k-1}}s_{i_k}.$$ In this case, we have that $$
s_{i_k}\ltimes s_{i_{k-1}}\ltimes s_{i_k}\ltimes w
=s_{i_k}s_{i_{k-1}}s_{i_k}ws_{i_{k-1}}.
$$
Since
$$\begin{aligned}
s_{i_{k-1}}s_{i_k}s_{i_{k-1}}ws_{i_{k-1}}&=s_{i_k}s_{i_{k-1}}s_{i_k}ws_{i_{k-1}}=s_{i_{k-1}}s_{i_k}ws_{i_{k-1}}s_{i_k}\\
&=s_{i_{k-1}}ws_{i_k}s_{i_{k-1}}s_{i_k}=s_{i_{k-1}}ws_{i_{k-1}}s_{i_k}s_{i_{k-1}},\end{aligned}
$$
it follows that $s_{i_k}s_{i_{k-1}}w=ws_{i_{k-1}}s_{i_k}$.
As a consequence, $$
s_{i_k}s_{i_{k-1}}s_{i_k}w=s_{i_k}s_{i_{k-1}}ws_{i_k}=ws_{i_{k-1}} ,
$$
and hence $$
s_{i_k}\ltimes s_{i_{k-1}}\ltimes s_{i_k}\ltimes w
=s_{i_k}s_{i_{k-1}}s_{i_k}ws_{i_{k-1}}=w.
$$
However, since $(s_{i_k},s_{i_{k-1}},s_{i_k},w)$ is a reduced sequence,
$$\ell(s_{i_k}\ltimes s_{i_{k-1}}\ltimes s_{i_k}\ltimes w)=\ell(w)+4.$$ We get a contradiction. Therefore, this case can not happen too.

\smallskip
{\it Case 7.} $s_{i_k}w=ws_{i_k}$, $s_{i_{k-1}}s_{i_k}w=s_{i_k}ws_{i_{k-1}}$ and $$
s_{i_k}s_{i_{k-1}}s_{i_k}w\neq s_{i_{k-1}}s_{i_k}ws_{i_k}.$$ In this case, we have that
$$
s_{i_k}\ltimes s_{i_{k-1}}\ltimes s_{i_k}\ltimes w=s_{i_k}s_{i_{k-1}}s_{i_k}ws_{i_k}=s_{i_k}s_{i_{k-1}}ws_{i_k}s_{i_k}=s_{i_k}s_{i_{k-1}}w.
$$
Since $(s_{i_k},s_{i_{k-1}},s_{i_k},w)$ is a reduced sequence, it follows from Lemma \ref{rankfunc} that $$
\ell(s_{i_k}s_{i_{k-1}}w)=\ell(s_{i_k}\ltimes s_{i_{k-1}}\ltimes s_{i_k}\ltimes w)=\ell(w)+4,
$$ which is impossible. Therefore, this case can not happen too.

\smallskip
{\it Case 8.} $s_{i_k}w=ws_{i_k}$, $s_{i_{k-1}}s_{i_k}w=s_{i_k}ws_{i_{k-1}}$ and $$
s_{i_k}s_{i_{k-1}}s_{i_k}w= s_{i_{k-1}}s_{i_k}ws_{i_k}.$$ In this case, we have that
$$
s_{i_k}\ltimes s_{i_{k-1}}\ltimes s_{i_k}\ltimes w =s_{i_k}s_{i_{k-1}}s_{i_k}w=s_{i_{k-1}}s_{i_k}ws_{i_k}=s_{i_{k-1}}ws_{i_k}s_{i_k}=s_{i_{k-1}}w.
$$
Since $(s_{i_k},s_{i_{k-1}},s_{i_k},w)$ is a reduced sequence, it follows from Lemma \ref{rankfunc} that $$
\ell(s_{i_{k-1}}w)=\ell(s_{i_k}\ltimes s_{i_{k-1}}\ltimes s_{i_k}\ltimes w)=\ell(w)+3,
$$ which is impossible. Therefore, this case can not happen too.\smallskip

This completes the proof of the statement 2) of the lemma.
\medskip

3) It suffices to show that $s_{i_{k}}\ltimes s_{i_k\pm 1}=s_{i_k\pm 1}\ltimes s_{i_{k}}$. By definition,
$$
s_{i_k\pm 1}\ltimes s_{i_{k}}=s_{i_{k}\pm 1}s_{i_k}s_{i_{k}\pm 1}=s_{i_k}s_{i_{k}\pm 1}s_{i_k}=s_{i_{k}}\ltimes s_{i_k\pm 1},
$$
as required.
\end{proof}

One of the important consequence of the above lemma is the following result, which will play important role in the proof of the main result of this paper.

\begin{cor} \label{braid0cor} Let $w\in I_\ast$. Let $(s_{i_1},\cdots,s_{i_k},w)$ be a reduced sequence. Suppose that $i_{k-2}=i_{k}=i_{k-1}\pm 1$, then either \begin{enumerate}
\item we have that $s_{i_k}w\neq ws_{i_k}$, $s_{i_{k-1}}s_{i_k}ws_{i_k}\neq s_{i_k}ws_{i_k}s_{i_{k-1}}$, $$
s_{i_k}s_{i_{k-1}}s_{i_k}ws_{i_k}s_{i_{k-1}}\neq s_{i_{k-1}}s_{i_k}ws_{i_k}s_{i_{k-1}}s_{i_k},$$
and $s_{i_{k-1}}w\neq ws_{i_{k-1}}$, $s_{i_k} s_{i_{k-1}} w s_{i_{k-1}}\neq s_{i_{k-1}} w s_{i_{k-1}} s_{i_k}$, $$
s_{i_{k-1}}s_{i_k}s_{i_{k-1}}ws_{i_{k-1}}s_{i_k}\neq s_{i_k}s_{i_{k-1}}ws_{i_{k-1}}s_{i_k}s_{i_{k-1}};
$$ or
\item we have that $s_{i_k}w\neq ws_{i_k}$, $s_{i_{k-1}}s_{i_k}ws_{i_k}\neq s_{i_k}ws_{i_k}s_{i_{k-1}}$, $$
s_{i_k}s_{i_{k-1}}s_{i_k}ws_{i_k}s_{i_{k-1}}= s_{i_{k-1}}s_{i_k}ws_{i_k}s_{i_{k-1}}s_{i_k},$$
and $s_{i_{k-1}}w=ws_{i_{k-1}}$, $s_{i_k}s_{i_{k-1}}w\neq s_{i_{k-1}}ws_{i_k}$, $$
s_{i_{k-1}}s_{i_k}s_{i_{k-1}}ws_{i_k}\neq s_{i_k}s_{i_{k-1}}ws_{i_k}s_{i_{k-1}};$$ or
\item we have that $s_{i_k}w=ws_{i_k}$, $s_{i_{k-1}}s_{i_k}w\neq s_{i_k}ws_{i_{k-1}}$, $$
s_{i_k}s_{i_{k-1}}s_{i_k}ws_{i_{k-1}}\neq s_{i_{k-1}}s_{i_k}ws_{i_{k-1}}s_{i_k},$$
and $s_{i_{k-1}}w\neq ws_{i_{k-1}}$, $s_{i_k}s_{i_{k-1}}ws_{i_{k-1}}\neq s_{i_{k-1}}ws_{i_{k-1}}s_{i_k}$, $$
s_{i_{k-1}}s_{i_k}s_{i_{k-1}}ws_{i_{k-1}}s_{i_k}=s_{i_k}s_{i_{k-1}}ws_{i_{k-1}}s_{i_k}s_{i_{k-1}} .
$$
\end{enumerate}
\end{cor}

\begin{proof} This follows from the proof of Lemma \ref{braid0}.
\end{proof}

\begin{cor} \label{braid0cor2} Let $w\in I_\ast$. Let $(s_a,s_b,w)$ be a reduced sequence. Suppose that $|a-b|>1$, then either \begin{enumerate}
\item $s_aw\neq ws_a$, $s_bs_a w s_a\neq s_aws_a s_b$ and $
s_bw\neq ws_b, s_as_b w s_b\neq s_b w s_bs_a;$ or
\item $s_aw= ws_a$, $s_bs_a w \neq s_aw s_b$ and $
s_bw\neq ws_b, s_as_b w s_b=s_b w s_bs_a;$ or
\item $s_aw\neq ws_a$, $s_bs_a w s_a=s_aws_a s_b$ and $
s_bw=ws_b, s_as_b w \neq s_b w s_a;$ or
\item $s_aw=ws_a$, $s_bs_a w =s_aw s_b$ and $
s_bw=ws_b, s_as_b w=s_b w s_a.$
\end{enumerate}
\end{cor}

\begin{proof} We only prove a) as the others can be proved in a similar manner and are left to the readers.

Suppose that $s_aw\neq ws_a$, $s_bs_a w s_a\neq s_aws_a s_b$. We first show that $s_bw\neq ws_b$. In fact, if
$s_bw=ws_b$ then (because $|a-b|>1$ implies that $s_as_b=s_bs_a$) $$
s_bs_a w s_a=s_as_bws_a=s_aws_bs_a=s_aws_as_b,
$$
which is a contradiction. This proves that $s_bw\neq ws_b$. Similarly, if $s_as_b w s_b=s_b w s_bs_a$, then we shall have that $s_bs_aws_b=s_as_b w s_b=s_b w s_bs_a=s_b w s_as_b$ which implies that $s_aw=ws_a$. We get a contradiction again. This proves that $s_as_b w s_b\neq s_b w s_bs_a$.
\end{proof}

In the rest of this section, we shall present some technical lemmas which will be used in the next section.

\begin{lem} \label{keyobser1} Let $1\leq i<n$ and $w\in\Sym_n$. Suppose that  $s_{i}s_{i+1}s_{i}w<s_{i+1}s_{i}w$. Then $s_{i+1}w<w$.
\end{lem}

\begin{proof} By assumption,
$$w^{-1}(i+1)=(w^{-1}s_is_{i+1})(i)>(w^{-1}s_is_{i+1})(i+1)=w^{-1}(i+2) .$$
It follows that $s_{i+1}w<w$.
\end{proof}

\begin{lem} \label{keyobser2} Let $w_2\in I_\ast$. Suppose that $\ell(s_{c+1}s_c w_2 s_{c}s_{c+1})=\ell(w_2)+4$, $s_{c}w_2\neq w_2s_{c}$, $s_{c+1}s_{c}w_2 s_c\neq s_c w_2 s_{c}s_{c+1}$ and $s_{c+1}w_2 s_{c}s_{c+1}<w_2 s_{c}s_{c+1}$. Then $s_{c+1}w_2<w_2$ and either $s_{c+1}w_2=w_2s_{c+1}$ or
$s_{c+1}w_2 s_{c+1}<s_{c+1}w_2$.
\end{lem}

\begin{proof} By assumption, \begin{equation}\label{cond1}
\ell(s_{c+1}w_2 s_{c}s_{c+1})=\ell(w_2 s_{c}s_{c+1})-1=\ell(w_2)+1 .\end{equation}

Suppose that $s_{c+1}w_2>w_2$. Then
$\ell(s_{c+1}w_2)=\ell(w_2)+1$ and (\ref{cond1}) imply that there are only the following two possibilities:

\smallskip
{\it Case 1.} $s_{c+1}w_2<s_{c+1}w_2 s_c>s_{c+1}w_2 s_cs_{c+1}$. In this case, we have that $$
s_{c+1}w_2(c+1)>s_{c+1}w_2(c)=s_{c+1}w_2s_c(c+1)>s_{c+1}w_2s_c(c+2)=s_{c+1}w_2(c+2) . $$
Now $w_2s_c>w_2<s_{c+1}w_2$ and $w_2=w_2^{-1}$ imply that $w_2(c)<w_2(c+1)<w_2(c+2)$. It follows that $$
c+1=w_2(c)<w_2(c+1)<w_2(c+2)=c+2 ,
$$
which is impossible.

\smallskip
{\it Case 2.} $s_{c+1}w_2>s_{c+1}w_2 s_c<s_{c+1}w_2 s_cs_{c+1}$. In this case, we have that $$
s_{c+1}w_2(c+1)<s_{c+1}w_2(c)=s_{c+1}w_2s_c(c+1)<s_{c+1}w_2s_c(c+2)=s_{c+1}w_2(c+2) . $$
Now $w_2s_c>w_2<s_{c+1}w_2$ and $w_2=w_2^{-1}$ imply that $w_2(c)<w_2(c+1)<w_2(c+2)$. It follows that $$
w_2(c)=c+1,\,\, w_2(c+1)=c+2<w_2(c+2) .
$$
Combining the above inequality with the assumption that $s_{c+1}w_2 s_{c}s_{c+1}<w_2 s_{c}s_{c+1}$ and $w_2=w_2^{-1}$, we can deduce that $$
c+1=s_{c+1}s_c w_2(c+1)>s_{c+1}s_c w_2(c+2)=w_2(c+2),
$$
which is again a contradiction. This proves the inequality $s_{c+1}w_2<w_2$. The remaining part of the lemma follows from Corollary \ref{length0} at once.
\end{proof}

\begin{lem} \label{keyobser2b} Let $w_2\in I_\ast$. Suppose that $\ell(s_{c+1}s_c w_2 s_{c})=\ell(w_2)+3$, $s_{c}w_2\neq w_2s_{c}$, $s_{c+1}s_{c}w_2 s_c= s_c w_2 s_{c}s_{c+1}$ and $s_{c+1}w_2 s_{c}<w_2 s_{c}$, then $s_{c+1}w_2<w_2$.
\end{lem}

\begin{proof} Suppose that $s_{c+1}w_2>w_2$. Then $s_{c+1}w_2>s_{c+1}w_2 s_{c}<w_2 s_{c}$. It follows that $$
s_{c+1}w_2(c+1)<s_{c+1}w_2(c).$$
Now $w_2s_c>w_2<s_{c+1}w_2$ and $w_2=w_2^{-1}$ imply that $w_2(c)<w_2(c+1)<w_2(c+2)$. It follows that $$
w_2(c)=c+1,\,\, w_2(c+1)=c+2,\,\,w_2(c+2)>c+2 .$$
Combining this with our assumption that $s_{c+1}w_2 s_{c}<w_2 s_{c}$ and $w_2=w_2^{-1}$, we can deduce that $$
c+2=s_c w_2(c+1)>s_c w_2(c+2)=w_2(c+2),
$$
which is a contradiction. So we must have that $s_{c+1}w_2<w_2$. This completes the proof of the lemma.
\end{proof}

\begin{lem} \label{simobser1} Let $w\in\Sym_n$ and $b\in\{1,2,\cdots,n\}$. Suppose that $w(t)<w(t+1)$, $\forall\, t <b$. If $w(b)\leq b$, then $w(i)=i$, for $1\leq i\leq b$.
\end{lem}
\begin{proof} By assumption, $1\leq w(1)<w(2)<w(3)<\cdots<w(b)\leq b$. It follows at once that $w(i)=i$, for any $1\leq i\leq b$.
\end{proof}

\begin{lem} \label{simobser1b} Let $w\in\Sym_n$ and $b\in\{1,2,\cdots,n\}$. Suppose that $w(t)<w(t+1)$, $\forall\, t \geq b$. If $w(b)\geq b$, then $w(j)=j$, for $b\leq j\leq n$.
\end{lem}
\begin{proof} By assumption, $b\leq w(b)<w(b+1)<w(b+2)<\cdots<w(n)\leq n$. It follows at once that $w(j)=j$, for any $b\leq j\leq n$.
\end{proof}
\bigskip

\section{Reduced $I_\ast$-expression and braid $I_{\ast}$-transformation}

A well-known classical fact of Matsumoto (\cite{Mathas}) says that any two reduced expressions for an element in $\Sym_n$ can be transformed into each other through a series of braid transformations. In Lemma \ref{braid0} we have shown that any braid $I_\ast$-transformations on reduced $I_\ast$-expression for a given $w\in I_\ast$ do not change the element $w$ itself. The following theorem says something more than this.

\begin{thm} \label{mainthm0} Let $w\in I_\ast$. Then any two reduced $I_\ast$-expressions for $w$ can be transformed into each other through a series of braid $I_\ast$-transformations.
\end{thm}

\begin{proof} We prove the theorem by induction on $\rho(w)$. Suppose that the theorem holds for any $w\in I_\ast$  with $\rho(w)\leq k$. Let $w\in I_\ast$ with $\rho(w)=k+1$. Let $(s_{c},s_{i_1},s_{i_2},\cdots,s_{i_k})$ and $(s_{b},s_{j_1},s_{j_2},\cdots,s_{j_k})$ be two reduced $I_\ast$-expressions for $w\in I_\ast$. We need to prove that \begin{equation}\label{goal}
(c,i_1,\cdots,i_k)\longleftrightarrow (b,j_1,\cdots,j_k).
\end{equation}
If $b=c$ then by induction hypothesis $(i_1,\cdots,i_k)\longleftrightarrow (j_1,\cdots,j_k)$ and hence (\ref{goal}) follows. Henceforth we assume that $b\neq c$.

By Lemma \ref{square}, $\rho(s_b\ltimes w)=\rho(s_{j_1}\ltimes s_{j_2}\ltimes\cdots\ltimes s_{j_k})=k<k+1$. It follows from Lemma \ref{rankfunc} that $\ell(s_bw)=\ell(w)-1$.
Equivalently, $$
\ell(s_b(s_{c}\ltimes s_{i_1}\ltimes s_{i_2}\ltimes\cdots\ltimes s_{i_k}))=\ell(s_{c}\ltimes s_{i_1}\ltimes s_{i_2}\ltimes\cdots\ltimes s_{i_k})-1 .
$$
Applying  Lemma \ref{rankfunc} again, we can deduce that $$
\rho(s_b\ltimes(s_{c}\ltimes s_{i_1}\ltimes s_{i_2}\ltimes\cdots\ltimes s_{i_k}))=k .
$$
We set $i_0:=c$. Applying Proposition \ref{exchange}, we get that $$
s_b\ltimes(s_{i_0}\ltimes s_{i_1}\ltimes s_{i_2}\ltimes\cdots\ltimes s_{i_k})=s_{i_0}\ltimes s_{i_1}\ltimes s_{i_2}\ltimes\cdots
\ltimes s_{i_{a-1}}\ltimes s_{i_{a+1}}\ltimes\cdots\ltimes s_{i_k}
$$
for some $0\leq a\leq k$. In particular, $s_{i_0}\ltimes s_{i_1}\ltimes s_{i_2}\ltimes\cdots
\ltimes s_{i_{a-1}}\ltimes s_{i_{a+1}}\ltimes\cdots\ltimes s_{i_k}=s_{j_1}\ltimes\cdots\ltimes s_{j_k}$.

Since $$
s_b\ltimes s_{i_0}\ltimes s_{i_1}\ltimes s_{i_2}\ltimes\cdots
\ltimes s_{i_{a-1}}\ltimes s_{i_{a+1}}\ltimes\cdots\ltimes s_{i_k}=s_{i_0}\ltimes s_{i_1}\ltimes s_{i_2}\ltimes\cdots\ltimes s_{i_k},
$$
it is clear that $(b,i_0,i_1,i_2,\cdots,i_{a-1},i_{a+1},\cdots,i_k)$ is a reduced $I_\ast$-expression for $w$.
\smallskip

We claim that for any $0\leq a\leq k$, \begin{equation}\label{goal2} (b,i_0,i_1,i_2,\cdots,i_{a-1},i_{a+1},\cdots,i_k)\longleftrightarrow
(i_0,i_1,i_2,\cdots,i_{a-1},i_a,i_{a+1},\cdots,i_k),\end{equation}
whenever $$
s_b\ltimes s_{i_0}\ltimes s_{i_1}\ltimes s_{i_2}\ltimes\cdots
\ltimes s_{i_{a-1}}\ltimes s_{i_{a+1}}\ltimes\cdots\ltimes s_{i_k}=s_{i_0}\ltimes s_{i_1}\ltimes s_{i_2}\ltimes\cdots\ltimes s_{i_k}
$$
holds. Once this is proved, we can deduce from induction hypothesis that $$(i_0,i_1,i_2,\cdots,i_{a-1},i_{a+1},\cdots,i_k)\longleftrightarrow (j_1,j_2,\cdots,j_k)$$ because $s_{i_0}\ltimes s_{i_1}\ltimes s_{i_2}\ltimes\cdots
\ltimes s_{i_{a-1}}\ltimes s_{i_{a+1}}\ltimes\cdots\ltimes s_{i_k}=s_{j_1}\ltimes\cdots\ltimes s_{j_k}$,
and hence $(b,i_0,i_1,i_2,\cdots,i_{a-1},i_{a+1},\cdots,i_k)\longleftrightarrow (b,j_1,j_2,\cdots,j_k)$. Composing these transformations, we prove (\ref{goal}). That is, $(i_0,i_1,i_2,\cdots,i_k)\longleftrightarrow (b,j_1,j_2,\cdots,j_k)$.

The remaining part of the argument is devote to the proof of (\ref{goal2}). First, we assume that $a>0$. If $|b-i_0|>1$, then by Lemma \ref{braid0}, $$
s_b\ltimes s_{i_0}\ltimes s_{i_1}\ltimes\cdots
\ltimes s_{i_{a-1}}\ltimes s_{i_{a+1}}\ltimes\cdots\ltimes s_{i_k}=s_{i_0}\ltimes s_b\ltimes s_{i_1}\ltimes\cdots
\ltimes s_{i_{a-1}}\ltimes s_{i_{a+1}}\ltimes\cdots\ltimes s_{i_k}.
$$
By induction hypothesis, $$(b,i_1,i_2,\cdots,i_{a-1},i_{a+1},\cdots,i_k)\longleftrightarrow (i_1,i_2,\cdots,i_{a-1},i_a,i_{a+1},\cdots,i_k),$$ and hence $$\begin{aligned}
&(b,i_0,i_1,i_2,\cdots,i_{a-1},i_{a+1},\cdots,i_k)\longleftrightarrow (i_0,b,i_1,i_2,\cdots,i_{a-1},i_{a+1},\cdots,i_k)\\
&\qquad \longleftrightarrow (i_0,i_1,i_2,\cdots,i_{a-1},i_a,i_{a+1},\cdots,i_k), \end{aligned}$$
where the second ``$\longleftrightarrow$" follows from induction hypothesis. So we are done in this case. Henceforth, we can assume that $|b-i_0|=1$.

Without loss of generality we can assume that $b=i_0+1$. The case when $b=i_0-1$ is exactly the same and is left to the readers. Let $w_1:=s_{i_1}\ltimes s_{i_2}\ltimes\cdots
\ltimes s_{i_{a-1}}\ltimes s_{i_{a+1}}\ltimes\cdots\ltimes s_{i_k}$. There are the following four possibilities:

\smallskip
{\it Case 1.} $s_{i_0}w_1\neq w_1 s_{i_0}$ and $s_b(s_{i_0}w_1s_{i_0})\neq (s_{i_0}w_1s_{i_0})s_b$. In this case, we have that $$
s_{i_0}\ltimes w_1=s_{i_0}w_1s_{i_0}, \quad s_b\ltimes s_{i_0}\ltimes w_1=s_b s_{i_0}w_1 s_{i_0}s_b .
$$
Since $\rho(s_{i_0}\ltimes(s_b\ltimes s_{i_0}\ltimes w_1))=\rho(s_{i_0}\ltimes(s_{i_0}\ltimes s_{i_1}\ltimes\cdots\ltimes s_{i_k}))=k$, it follows that
$s_{i_0}(s_b\ltimes s_{i_0}\ltimes w_1)<s_b\ltimes s_{i_0}\ltimes w_1$. Applying Lemma \ref{keyobser1}, we can deduce that $s_{b}w_1 s_{i_0}s_b<w_1 s_{i_0}s_b$.
Now we are in a position to apply Lemma \ref{keyobser2} so that we can deduce that $s_{b}w_1<w_1$. Applying Corollary \ref{deletion2}, we see that $w_1=s_b\ltimes w_2$ with $w_2\in I_\ast$ and $(s_b,w_2)$ being reduced. Thus by Lemma \ref{braid0}, $$
s_b\ltimes s_{i_0}\ltimes w_1=s_b\ltimes s_{i_0}\ltimes s_b\ltimes w_2=s_{i_0}\ltimes s_b\ltimes s_{i_0}\ltimes w_2 .
$$
On the other hand, recall that $$
s_{i_0}\ltimes s_{i_1}\ltimes\cdots\ltimes s_k=s_b\ltimes s_{i_0}\ltimes w_1 .
$$
It follows that $$
s_{i_1}\ltimes s_{i_2}\ltimes\cdots\ltimes s_k=s_b\ltimes s_{i_0}\ltimes w_2 .
$$
Now using induction hypothesis, we see that $(i_1,i_2,\cdots,i_k)\longleftrightarrow (b,i_0,w_2)$, and hence $(i_0,i_1,i_2,\cdots,i_k)\longleftrightarrow (i_0,b,i_0,w_2)$. Composing this with the transformation
$(i_0,b,i_0)\longleftrightarrow(b,i_0,b)$, we can get that $$(i_0,i_1,i_2,\cdots,i_k)\longleftrightarrow (b,i_0,b,w_2)=(b,i_0,w_1).$$ It follows that $$(i_0,i_1,i_2,\cdots,i_{a-1},i_a,i_{a+1},\cdots,i_k)\longleftrightarrow
(b,i_0,i_1,i_2,\cdots,i_{a-1},i_{a+1},\cdots,i_k) $$ as required.

\smallskip
{\it Case 2.} $s_{i_0}w_1\neq w_1 s_{i_0}$ and $s_b(s_{i_0}w_1s_{i_0})=(s_{i_0}w_1s_{i_0})s_b$. In this case, we have that $$
s_{i_0}\ltimes w_1=s_{i_0}w_1s_{i_0}, \quad s_b\ltimes s_{i_0}\ltimes w_1=s_b s_{i_0}w_1 s_{i_0}=s_{i_0}w_1 s_{i_0}s_b .
$$
Since $\rho(s_{i_0}\ltimes(s_b\ltimes s_{i_0}\ltimes w_1))=\rho(s_{i_0}\ltimes(s_{i_0}\ltimes s_{i_1}\ltimes\cdots\ltimes s_{i_k}))=k$, it follows that
$s_{i_0}(s_b\ltimes s_{i_0}\ltimes w_1)<s_b\ltimes s_{i_0}\ltimes w_1$. That is, $s_{i_0}s_bs_{i_0}w_1s_{i_0}<s_bs_{i_0}w_1s_{i_0}$. Applying Lemma \ref{keyobser1}, we can deduce that $s_{b}w_1 s_{i_0}<w_1 s_{i_0}$.
Once again we are in a position to apply Lemma \ref{keyobser2b} so that we can deduce that $s_{b}w_1<w_1$. Now one can repeat the same argument used in the proof of Case 1 to complete the remaining proof in this case.

\smallskip
{\it Case 3.} $s_{i_0}w_1=w_1 s_{i_0}$ and $s_b(s_{i_0}w_1)=(s_{i_0}w_1)s_b$. In this case, we have that $$
s_{i_0}\ltimes w_1=s_{i_0}w_1=w_1s_{i_0}, \quad s_b\ltimes s_{i_0}\ltimes w_1=s_b s_{i_0}w_1=w_1s_{i_0}s_b .
$$
Since $\rho(s_{i_0}\ltimes(s_b\ltimes s_{i_0}\ltimes w_1))=\rho(s_{i_0}\ltimes(s_{i_0}\ltimes s_{i_1}\ltimes\cdots\ltimes s_{i_k}))=k$, it follows that
$s_{i_0}(s_b\ltimes s_{i_0}\ltimes w_1)<s_b\ltimes s_{i_0}\ltimes w_1$. That is, $s_{i_0}s_bs_{i_0}w_1<s_bs_{i_0}w_1$. Applying Lemma \ref{keyobser1}, we can deduce that $s_{b}w_1 <w_1 $. Hence
$w_1s_b<w_1$ as $w_1=w_1^{-1}$. On the other hand, since $(s_b,s_{i_0},w_1)$ is a reduced sequence, by Lemma \ref{rankfunc},
$$\ell(s_{i_0}w_1s_b)=\ell(s_b\ltimes s_{i_0}\ltimes w_1)=\ell(w_1)+2, $$ which is a contradiction. Therefore, this case can not happen.

\smallskip
{\it Case 4.} $s_{i_0}w_1=w_1 s_{i_0}$ and $s_b(s_{i_0}w_1)\neq (s_{i_0}w_1)s_b$. In this case, we have that $$
s_{i_0}\ltimes w_1=s_{i_0}w_1=w_1s_{i_0}, \quad s_b\ltimes s_{i_0}\ltimes w_1=s_b s_{i_0}w_1 s_{b}=s_b w_1s_{i_0}s_b .
$$
Since $\rho(s_{i_0}\ltimes(s_b\ltimes s_{i_0}\ltimes w_1))=\rho(s_{i_0}\ltimes(s_{i_0}\ltimes s_{i_1}\ltimes\cdots\ltimes s_{i_k}))=k$, it follows that
$s_{i_0}(s_b\ltimes s_{i_0}\ltimes w_1)<s_b\ltimes s_{i_0}\ltimes w_1$. That is, $s_{i_0}s_bs_{i_0}w_1s_{b}<s_bs_{i_0}w_1s_{b}$. Applying Lemma \ref{keyobser1}, we can deduce that $s_{b}w_1 s_{b}<w_1 s_{b}$. Note that $(s_b,s_{i_0},w_1)$ is a reduced sequence. Applying Lemma \ref{rankfunc} we know that $$
\ell(s_b s_{i_0}w_1s_b)=\ell(s_b w_1s_{i_0}s_b)=\ell(w_1)+3 . $$ It follows that $w_1s_b>w_1<s_b w_1$. Now we have that $s_{b}w_1 s_{b}<w_1 s_{b}>w_1$. Applying Corollary \ref{length0}, we get that $s_b w_1=w_1s_b$. Since
$$s_{i_0}s_bw_1=s_{i_0}w_1s_b\neq s_bs_{i_0}w_1=s_bw_1s_{i_0},
$$
it follows that
$$
s_b\ltimes s_{i_0}\ltimes w_1=s_b s_{i_0}w_1 s_{b}=s_b s_{i_0}s_b w_1=s_{i_0}s_bs_{i_0}w_1=s_{i_0}s_b w_1 s_{i_0}=s_{i_0}\ltimes s_b\ltimes w_1.$$
By induction hypothesis, $(i_0,b,w_1)\longleftrightarrow (i_0,i_1,i_2,\cdots,i_k)$. It remains to show that  $(b,i_0,w_1)\longleftrightarrow (i_0,b,w_1)$.

\smallskip
Recall that $i_0=c=b-1$. Since $s_{b}w_1 s_{b}<w_1 s_{b}>w_1$ and $w_1=w_1^{-1}$ imply that $w_1(b)<w_1(b+1)$ and $s_{b}w_1(b)>s_{b}w_1(b+1)$, it follows that $w_1(b)=b$, $w_1(b+1)=b+1$. Since $s_{b-1}w_1=w_1s_{b-1}$, it follows that $s_{b-1}w_1(b-1)=w_1s_{b-1}(b-1)=w_1(b)=b$, and hence $w_1(b-1)=b-1$. There are the following two subcases:

\smallskip
{\it Subcase 1.} There exists some $t\leq b-3$ or $t\geq b+2 $ such that $s_t w_1<w_1$. In this case, we obtain that $w_1=s_t\ltimes w_2$ with $w_2\in I_\ast$ and $(s_t,w_2)$ being reduced. It is easy to see that
$$
(s_b,s_{b-1},w_1)\longleftrightarrow (s_b,s_{b-1},s_t,w_2)\longleftrightarrow (s_b,s_t,s_{b-1},w_2) \longleftrightarrow (s_t,s_b,s_{b-1},w_2),$$
$$
(s_{b-1},s_b,w_1)\longleftrightarrow (s_{b-1},s_b,s_t,w_2)\longleftrightarrow (s_{b-1},s_t,s_b,w_2)\longleftrightarrow (s_t,s_{b-1},s_b,w_2).$$
By induction hypothesis, we have that $(s_t,s_b,s_{b-1},w_2)\longleftrightarrow(s_t,s_{b-1},s_b,w_2)$ as $\rho(s_b\ltimes s_{b-1}\ltimes w_2)=k$. Therefore, $(s_b,s_{b-1},w_1)\longleftrightarrow(s_{b-1},s_b,w_1)$ as required.

\smallskip
{\it Subcase 2.} For any $t\leq b-3$ or $t\geq b+2 $, we always have that $s_t w_1>w_1$ and hence $w_1(t)<w_1(t+1)$. In this case, assume that $w_1(b-2)\leq b-2$. Using Lemma \ref{simobser1} we can deduce that $w_1(i)=i$ for any $1\leq i\leq b-2$. Hence $w(b+2)\geq b+2 $, which implies that $w_1(j)=j$ for any $b+2\leq j\leq n$ (by Lemma \ref{simobser1b} again). Therefore, $w_1=1$. Clearly we get $(s_b,s_{b-1})\longleftrightarrow(s_{b-1},s_b)$ as required.

Therefore it suffices to consider the situation when $w_1(b-2)> b-2$. By a similar argument as in the last paragraph, we can only consider the situation when $w_1(b+2)< b+2$.

If $b+2\leq w_1(t)<w_1(t+1)$ for some $t\leq b-3$, then we must have that $w_1(t+1)=w_1(t)+1$ because otherwise $b+2\leq w_1(t)<z<w_1(t+1)$ implies that $t<w(z)<t+1$, which is impossible.

Now suppose that $w_1(b+2)=b-1-r$, for some $r>0$. Then the discussion in the last paragraph and the fact that $w_1^2=1$ imply that $$\begin{aligned}
& w_1(b-1-r)=b+2,w_1(b-r)=b+3,\cdots,w_1(b-2)=b+1+r,\\
& w_1(b+2)=b-1-r,w_1(b+3)=b-r,\cdots,w_1(b+1+r)=b-2.\end{aligned}$$
In particular, $w_1(b-2-r)\leq b-2-r$ , $w_1(b+2-r)\geq b+2-r$. Applying Lemma \ref{simobser1} and Lemma \ref{simobser1b} we see that $w_1(i)=i$ for any $i\leq b-2-r$ or $i\geq b+2+r$. The permutation $w_1$ can be depicted in Figure 1 as follows:
\begin{center}
\begin{tikzpicture}[scale=0.6]
    \useasboundingbox (-1,-1) rectangle (20,3);
    \tikzstyle{every node}=[font=\tiny]
    \node at(0,0)[below]{$1$};
    \node at(3,0)[below,xshift=-3mm]{$b-2-r$};
    \node at(4,0)[below,xshift=3mm]{$b-1-r$};
    \node at(7,0)[below,xshift=-1.5mm]{$b-2$};
    \node at(8,0)[below]{$b-1$};
    \node at(9,0)[below]{$b$};
    \node at(10,0)[below]{$b+1$};
    \node at(11,0)[below,xshift=1.5mm]{$b+2$};
    \node at(14,0)[below,xshift=-3mm]{$b+1+r$};
    \node at(15,0)[below,xshift=3mm]{$b+2+r$};
    \node at(18,0)[below]{$n$};
    \node at(0,2)[above]{$1$};
    \node at(3,2)[above,xshift=-3mm]{$b-2-r$};
    \node at(4,2)[above,xshift=3mm]{$b-1-r$};
    \node at(7,2)[above,xshift=-1.5mm]{$b-2$};
    \node at(8,2)[above]{$b-1$};
    \node at(9,2)[above]{$b$};
    \node at(10,2)[above]{$b+1$};
    \node at(11,2)[above,xshift=1.5mm]{$b+2$};
    \node at(14,2)[above,xshift=-3mm]{$b+1+r$};
    \node at(15,2)[above,xshift=2.5mm]{$b+2+r$};
    \node at(18,2)[above]{$n$};
    \foreach \i in {0,3,4,7,8,9,10,11,14,15,18}
      {
        \path(\i,0) coordinate (x\i) (\i,2) coordinate (y\i);
        \fill(x\i) circle(3pt);
        \fill(y\i) circle(3pt);
      }
    \foreach \i in {0,3,8,9,10,15,18}
      {
         \draw (x\i)--(y\i);
      }
     \draw (4,0)--(11,2)(7,0)--(14,2)(11,0)--(4,2)(14,0)--(7,2);
     \foreach \i in {0,4,11,15}
       { \draw[thick,dotted]
          (\i,0)--(\i+3,0)(\i,2)--(\i+3,2) ;
       }
\end{tikzpicture}

Figure $1$
\end{center}

Since $w_1(b-2)=b+1+r>b-1=w_1(b-1)$ implies that $s_{b-2}w_1<w_1$, it follows from Corollary \ref{deletion2}
that $w_1=s_{b-2}\ltimes w_2$ with $w_2\in I_\ast$ and $(s_{b-2},w_2)$ being reduced. Since $s_{b-2}w_1(b-1)=s_{b-2}(b-1)=b-2$ and $w_1s_{b-2}(b-1)=w_1(b-2)=b+r+1$, we get that $s_{b-2}w_1\neq w_1s_{b-2}$ and hence $w_2=s_{b-2}\ltimes s_{b-2}\ltimes w_2=s_{b-2}\ltimes w_1=s_{b-2}w_1s_{b-2}$. Furthermore, since $w_2(b-1)=s_{b-2}w_1s_{b-2}(b-1)=b+1+r$ and $w_2(b)=s_{b-2}w_1s_{b-2}(b)=b$, it follows that $s_{b-1}w_2<w_2$ and $w_2=s_{b-1}\ltimes w_3$ with $w_3\in I_\ast$ and $(s_{b-1},w_3)$ being reduced. Since $s_{b-1}w_2(b)=b-1$ and $w_2s_{b-1}(b)=b+r+1$, we have that $s_{b-1}w_2\neq w_2s_{b-1}$ and hence $w_3=s_{b-1}\ltimes s_{b-1}\ltimes w_3=s_{b-1}\ltimes w_2=s_{b-1}w_2s_{b-1}$.

We proceed further by similar argument. Since $w_3(b)=s_{b-1}s_{b-2}w_1s_{b-2}s_{b-1}(b)=b+1+r$ and $w_3(b+1)=s_{b-1}s_{b-2}w_1s_{b-2}s_{b-1}(b+1)=b+1$, it follows that $s_{b}w_3<w_3$ and $w_3=s_{b}\ltimes w_4$ with $w_4\in I_\ast$ and $(s_{b},w_4)$ being reduced. Now we obtain that $w_1=s_{b-2}\ltimes s_{b-1}\ltimes s_{b} \ltimes w_4$ and $(s_{b-2},s_{b-1},s_{b},w_4)$ is reduced. Now by induction hypothesis and
Lemma \ref{braid0},
$$\begin{aligned}
&(s_b,s_{b-1},w_1)\longleftrightarrow(s_b,s_{b-1},s_{b-2},s_{b-1},s_b,w_4)\longleftrightarrow\\
&\qquad (s_b,s_{b-2},s_{b-1},s_{b-2},s_b,w_4) \longleftrightarrow (s_{b-2},s_b,s_{b-1},s_{b-2},s_b,w_4)\end{aligned}
$$
and $$\begin{aligned} &(s_{b-1},s_b,w_1)\longleftrightarrow(s_{b-1},s_b,s_{b-2},s_{b-1},s_b,w_4)\longleftrightarrow (s_{b-1},s_{b-2},s_b,s_{b-1},s_b,w_4)\\
&\qquad \longleftrightarrow(s_{b-1},s_{b-2},s_{b-1},s_b,s_{b-1},w_4) \longleftrightarrow (s_{b-2},s_{b-1},s_{b-2},s_b,s_{b-1},w_4).\end{aligned} $$
Once again by induction hypothesis, $$
(s_{b-2},s_b,s_{b-1},s_{b-2},s_b,w_4) \longleftrightarrow (s_{b-2},s_{b-1},s_{b-2},s_b,s_{b-1},w_4), $$
This completes the proof of $(s_{b-1},s_b,w_1) \longleftrightarrow (s_b,s_{b-1},w_1)$. Hence we complete the proof of  (\ref{goal2}) when $a>0$.

\medskip
It remains to prove (\ref{goal2}) when $a=0$. In this case, we want to show that $$
(b,i_1,i_2,\cdots,i_k)\longleftrightarrow (c,i_1,i_2,\cdots,i_k).$$

We set $w_1:=s_{i_1}\ltimes s_{i_2}\ltimes\cdots\ltimes s_{i_k}$. Since $s_b\ltimes w_1=s_c\ltimes w_1 $ and $s_b\neq s_c$, it follows that $s_b\ltimes w_1=s_b w_1 s_b$ and $s_c\ltimes w_1=s_c w_1 s_c$, which imply that $w_1 s_b>w_1<w_1 s_c$ and hence that $w_1(b)<w_1(b+1)$, $w_1(c)<w_1(c+1)$. Since $\rho(s_c\ltimes (s_b\ltimes w_1))=\rho(w_1)=k$, it follows that $s_c (s_b\ltimes w_1)<s_b\ltimes w_1$. That is, $s_cs_bw_1s_b<s_bw_1s_b$, which forces that $s_bw_1s_b(c)>s_bw_1s_b(c+1)$.

We claim that $|b-c|>1$. Suppose this is not case. Without loss of generality we can assume that $c=b+1$. We have proved that $w_1(b)<w_1(b+1)<w_1(b+2)$ and $s_bw_1(b)=s_bw_1s_b(b+1)>s_bw_1s_b(b+2)=s_bw_1(b+2)$. It follows that $w_1(b)=b$, $w_1(b+2)=b+1$, a contradiction. This proves the claim that $|b-c|>1$.

Since $|b-c|>1$, $s_bw_1(c)=s_bw_1s_b(c)>s_bw_1s_b(c+1)=s_bw_1(c+1)$ and $w_1(c)<w_1(c+1)$. Therefore we can deduce that $w_1(c)=b$ and $w_1(c+1)=b+1$. There are two possibilities:

\smallskip
{\it Case 1.} $|b-c|=2$. Without loss of generality we can assume that $c=b+2$. In this case, we consider the following two subcases:

\smallskip
{\it Subcase 1.} There exists some $t\leq b-2$ or $t\geq b+4 $ such that $s_t w_1<w_1$. In this subcase, by Corollary \ref{deletion2}, we see that $w_1=s_t\ltimes w_2$ with
$w_2\in I_\ast$ and $(s_t,w_2)$ being reduced. By induction hypothesis and Lemma \ref{braid0}, we see that
$$\begin{aligned}
&(s_b,w_1)\longleftrightarrow(s_b,s_t,w_2)\longleftrightarrow (s_t,s_b,w_2),\\
&(s_{b+2},w_1)\longleftrightarrow(s_{b+2},s_t,w_2)\longleftrightarrow (s_t,s_{b+2},w_2),\\
&(s_t,s_b,w_2)\longleftrightarrow(s_t,s_{b+2},w_2).\end{aligned}$$
It follows that $(s_b,w_1)\longleftrightarrow(s_{b+2},w_1)$ as required.

\smallskip
{\it Subcase 2.} For any $t\leq b-2$ or $t\geq b+4 $, we always have $s_t w_1>w_1$ and hence $w_1(t)<w_1(t+1)$. In this subcase, we assume first that $w_1(b-1)\leq b-1$. Using Lemma \ref{simobser1} we have that $w_1(i)=i$ for any $1\leq i\leq b-1$. As a result, $w(b+4)\geq b+4$, which (by Lemma \ref{simobser1b}) in turn forces $w_1(j)=j$ for any $b+4\leq j\leq n$ . Therefore, $$w_1=(b,b+2)(b+1,b+3)=s_{b+1}s_bs_{b+2}s_{b+1}=s_{b+1}\ltimes s_b\ltimes s_{b+2}.$$
Consequently, by induction hypothesis and Lemma \ref{braid0}, we can deduce that $$\begin{aligned}
&(s_b,w_1)\longleftrightarrow(s_b,s_{b+1},s_{b},s_{b+2})\longleftrightarrow(s_{b+1},s_b,s_{b+1},s_{b+2})
\longleftrightarrow\\
&\qquad (s_{b+1},s_b,s_{b+2},s_{b+1})\longleftrightarrow(s_{b+1},s_{b+2},s_b,s_{b+1})\longleftrightarrow
(s_{b+1},s_{b+2},s_{b+1},s_b)\\
&\qquad\qquad \longleftrightarrow(s_{b+2},s_{b+1},s_{b+2},s_b)
\longleftrightarrow(s_{b+2},s_{b+1},s_b,s_{b+2})\longleftrightarrow(s_{b+2},w_1)\end{aligned}$$
as required.

\smallskip
To finish the proof in Subcase 2, we only need to consider the situation when $w_1(b-1)> b-1$. By a similar (and symmetric) argument as in the last paragraph, we can only consider the case when $w_1(b+4)< b+4$.

Recall that $w_1^2=1$ and $w_1(t)\leq w_1(t+1)$ for any $t\leq b-2$. If $w_1(t)\geq b+4$ for some $t\leq b-2$, then $w_1(t+1)=w_1(t)+1$ because otherwise $w_1(t)<q<w_1(t+1)$ implies that $t<w(q)<t+1$ which is impossible. Recall also that $w_1(b)=c=b+2$, $w_1(b+2)=w_1(c)=b$ and $w_1(b)<w_1(b+1)$, $w_1(b+2)<w_1(b+3)$. It follows that
$w_1(b-1)\not\in\{b,b+1,b+2,b+3\}$. In particular, $w_1(b-1)\geq b+4$. We write $w_1(b-1)=b+3+r$ for some $r>0$.
Since $w_1(b-1)=b+3+r>b+2=w_1(b)$ implies that $s_{b-1}w_1<w_1$, it follows that $w_1=s_{b-1}\ltimes w_2$ with $w_2\in I_\ast$ and $(s_{b-1},w_2)$ being reduced. Now we obtain that $$(s_b,w_1)\longleftrightarrow(s_b,s_{b-1},w_2)$$
and $$(s_{b+2},w_1)\longleftrightarrow(s_{b+2},s_{b-1},w_2)\longleftrightarrow (s_{b-1},s_{b+2},w_2).$$
By assumption, $$
s_b\ltimes s_{b-1}\ltimes w_2=s_b\ltimes w_1=s_{b+2}\ltimes w_1=s_{b-1}\ltimes s_{b+2}\ltimes w_2 .
$$
So we are in a position to apply (\ref{goal2}) in the case when $a=1$ (which we have already proved). Therefore, we are done in this case.

\smallskip
{\it Case 2.} $|b-c|>2$. Without loss of generality we can assume that $c>b+2$. There are two subcases:

\smallskip
{\it Subcase 1.} There exists some $t\leq b-2$ or $b+2 \leq t\leq c-2 $ or $t\geq c+2$ such that $s_t w_1<w_1$. In this case, we obtain that $w_1=s_t\ltimes w_2$ with
$w_2\in I_\ast$ and $(s_t,w_2)$ being reduced. By Lemma \ref{braid0},
$$\begin{aligned}
&(s_b,w_1)\longleftrightarrow(s_b,s_t,w_2)\longleftrightarrow (s_t,s_b,w_2),\\
&(s_{c},w_1)\longleftrightarrow(s_{c},s_t,w_2)\longleftrightarrow (s_t,s_{c},w_2).\end{aligned}$$
Note that $(s_t,s_b,w_2)\longleftrightarrow(s_t,s_{c},w_2)$ by induction hypothesis. As a result, we get that $(s_b,w_1)\longleftrightarrow(s_{c},w_1)$ as required.

\smallskip
{\it Subcase 2.} For any $t\leq b-2$ or $b+2 \leq t\leq c-2 $ or $t\geq c+2$, we always have that $s_t w_1>w_1$. That is, $w_1(t)<w_1(t+1)$.

Recall that $w_1^2=1$ and we have shown that $$
w_1(b)=c,\,\,w_1(c)=b,\,\,w_1(b+1)=c+1,\,\,w_1(c+1)=b+1 .
$$
We claim that either $s_{c-1}w_1<w_1$ or $s_{b+1}w_1<w_1$. Suppose this is not the case. That says, $s_{c-1}w_1>w_1$ and $s_{b+1}w_1>w_1$. Then we can deduce that $$c+1=w_1(b+1)<w_1(b+2)<w_1(b+3)<\cdots <w_1(c-2)<w_1(c-1)<w_1(c)=b,$$ which is a contradiction. This proves our claim.

Suppose that $s_{b+1}w_1<w_1$. Then $w_1=s_{b+1}\ltimes w_2$ with $w_2\in I_\ast$ and $(s_{b+1},w_2)$ being reduced. Now we obtain that $$(s_b,w_1)\longleftrightarrow(s_b,s_{b+1},w_2)$$ and $$(s_c,w_1)\longleftrightarrow(s_{c},s_{b+1},w_2)\longleftrightarrow (s_{b+1},s_{c},w_2).$$
By assumption, $$
s_b\ltimes s_{b+1}\ltimes w_2=s_b\ltimes w_1=s_c\ltimes w_1=s_{b+1}\ltimes s_c\ltimes w_2 .
$$
So we are in a position to apply (\ref{goal2}) in the case when $a=1$ (which we have already proved). Therefore, we are done in this case. By a similar argument, we can reduce the assertion when $s_{c-1}w_1<w_1$ to a statement of the form (\ref{goal2}) with $a=1$ (which we have already proved). Therefore, we complete the proof of the theorem.
\end{proof}

\bigskip
\section{The lower bound of the dimension of $\mathcal{H}^{\Q(u)}X_{\emptyset}$ when $\ast=\text{id}$ and $W=\Sym_n$}

Recall that $\ast=\text{id}$ and $W=\Sym_n$. In this section we shall prove that the dimension of $\mathcal{H}^{\Q(u)}X_{\emptyset}$ is bigger or equal than the number of involutions in $\Sym_n$.

It is well-known that $\mathcal{H}^{\Q(u)}$ is a split semisimple $\Q(u)$-algebra. To recall some well-known results in its representation theory, we need some combinatorics.

A composition of $n$ is a sequence of non-negative integers $\lam=(\lam_1,\lam_2,\cdots,\lam_r)$ such that $\sum_{i=1}^r\lam_i=n$. The composition $\lam=(\lam_1,\lam_2,\cdots,\lam_r)$ is called a partition if
$\lam_1\geq\lam_2\geq\cdots\geq\lam_r$. We use $\mathcal{P}_n$ to denote the set of partitions of $n$. Let $\lam\in\mathcal{P}_n$. The Young diagram of $\lam$ is the set
$$[\lam]=\bigl\{(a,c)\bigm|1\le c\le \lambda_a, a\ge1\bigr\}. $$
A $\lam$-tableau is a bijective map $\t: [\lam]\rightarrow\{1,2,\dots,n\}$. If~$\t$ is a $\lam$-tableau then set $\Shape(\t)=\blam$. A $\lam$-tableau $\t$ is said to be row (column) standard if the numbers $1,2,\dots,n$ increase
along the rows (columns) of $\t$, and standard if $\t$ is both row and column standard.
Let $\Std(\lam)$ be the set of standard $\lam$-tableaux.

Let $\lam,\mu\in\mathcal{P}_n$. If $\t$ is a standard $\lam$-tableau, then let $\t\downarrow_k$ be the
subtableau of $\t$ labeled by $1,\dots,k$ in $\t$. If
$\s\in\Std(\lam)$ and $\t\in\Std(\mu)$ then $\s$
dominates $\t$, and we write $\s\unrhd\t$, if
$\Shape(\s\downarrow_k)\unrhd\Shape(\t\downarrow_k)$, for $k=1,\dots,n$.
We write $\s\rhd\t$ if $\s\unrhd\t$ and $\s\ne\t$. Let $\tlam$ be the unique standard $\lam$-tableau such that
$\tlam\unrhd\t$ for all $\t\in\Std(\lam)$. Then $\tlam$ has the
numbers $1,\dots,n$ entered in order, from left to right and then top
to bottom along the rows of $\lam$. Let $\tllam$ be the unique standard $\lam$-tableau such that
$\tlam\unlhd\t$ for all $\t\in\Std(\lam)$. Then $\tllam$ has the
numbers $1,\dots,n$ entered in order, from top
to bottom and then left to right along the columns of $\lam$. If $\lambda=(\lambda_1,\lambda_2,\dots)$ is a partition then its conjugate is the partition $$\lambda'=(\lambda'_1,\lambda'_2,\dots),$$ where
$\lambda'_i=\#\{j\ge1|\lambda_j\ge i\}$. If $\t$ is a standard $\lambda$-tableau let $\t'$ be the standard $\lambda'$-tableau given by $\t'(r,c)=\t(c,r)$.

It is well-known that $\mathcal{H}^{\Q(u)}$ is a split semisimple algebra over $\Q(u)$. Following\,\, \cite[2.4]{M:gendeg}, let $\{f_{\s\t}|\s,\t\in\Std(\lam),\lam\in\mathcal{P}_n\}$ be the seminormal basis of $\mathcal{H}^{\Q(u)}$. By definition, for any $\lam\in\mathcal{P}_n$, $\s,\t,\u,\v\in\Std(\lam)$, we have that $$
f_{\s\t}f_{\u\v}=\delta_{\t\u}\gamma_{\t}f_{\s\v},
$$
where $\gamma_{\t}\in\Q(u)^{\times}$ is a nonzero scalar (which can be written down explicitly).
Furthermore, $\mathcal{H}^{\Q(u)}f_{\s\t}\cong\mathcal{H}^{\Q(u)}f_{\s\tlam}$ is a simple left $\mathcal{H}^{\Q(u)}$-module. We denote this module by $V_{\Q(u)}^{\lam}$. Then $\{V_{\Q(u)}^{\lam}|\lam\in\mathcal{P}_n\}$ is a complete set of pairwise non-isomorphic simple left $\mathcal{H}^{\Q(u)}$-modules.

For any subset $J\subseteq \{1,2,\cdots,n\}$, we use $\Sym_J$ to denote the standard Young subgroup of $\Sym_n$ generated by $\{s_i|i,i+1\in J\}$. Let $\lam\in\mathcal{P}_n$. The symmetric group $\Sym_n$ acts on the set of $\lam$-tableaux from the left hand-side. If $\t$ is a $\lam$-tableau with $\lam\in\mathcal{P}_n$, and $w\in\Sym_n$, we also define $$
\t w:=w^{-1}\t .
$$

Let $\lam$ be a composition of $n$. Let $\Sym_{\lam}$ be the row stabilizer of $\tlam$, which is the standard Young subgroup of $\Sym_n$ corresponding to $$
I_{\lam}:=\{1,2,\cdots,\lam_1\}\sqcup\{\lam_1+1,\lam_1+2,\cdots,\lam_1+\lam_2\}\sqcup\cdots .
$$
Let $\mathcal{H}_{u}(\Sym_{\lam})$ be the subalgebra of $\mathcal{H}_{u}$ generated by $\{T_i|s_i\in\Sym_{\lam}\}$. If $\t\in\Std(\lam)$ let $d(\t)$ be the permutation in $\Sym_n$ such that $\t=\tlam d(\t)$. Let $$\mathcal{D}_{\lam}:=\{d\in\Sym_n|\text{$\tlam d$ is row standard}\}.$$ Then $\mathcal{D}_\lam$ is the set of minimal length distinguished right coset representatives of $\Sym_{\lam}$ in $\Sym_n$. For any $\lam,\mu\in\mathcal{P}_n$, we set $\mathcal{D}_{\lam,\mu}:=\mathcal{D}_\lam\cap \mathcal{D}_\mu^{-1}$. Then $\mathcal{D}_{\lam,\mu}$ is the set of minimal length distinguished double coset representatives of $(\Sym_{\lam},\Sym_{\mu})$ in $\Sym_n$. Let $w_\lam\in\Sym_n$ such that $\tlam w_{\lam}=\tllam$. Then $w_{\lam}\in\mathcal{D}_{\lam,\lam'}$.

\begin{lem} \label{lb} Let $\lam\in\mathcal{P}_n$. Then $f_{\t\tlam}X_{\emptyset}f_{\tllam\tllam}\neq 0$ for any $\t\in\Std(\lam)$. In particular, we have that $$
\dim_{\Q(u)}\mathcal{H}^{\Q(u)}X_{\emptyset}\geq\sum_{\lam\in\mathcal{P}_n}\#\Std(\lam) .
$$
\end{lem}

\begin{proof} It suffices to show that $f_{\tlam\tlam}X_{\emptyset}f_{\tllam\tllam}\neq 0$ because
$f_{\t\tlam}f_{\tlam\tlam}=\gamma_{\tlam}f_{\t\tlam}$ for some $\gamma_{\tlam}\in \Q(u)^{\times}$.

By\,\, \cite[Lemma 1.1]{DJ1}, for any $w\in\Sym_n$, there exists a unique element $d\in\Sym_{\lam}w\Sym_{\mu}$ such that
$$
d\in\mathcal{D}_{\lam,\mu},\,\,w=w_1dw_2,\,w_1\in\Sym_{\lam}, w_2\in\mathcal{D}_{\lam d\cap\mu}\cap\Sym_{\mu},\,
\ell(w)=\ell(w_1)+\ell(d)+\ell(w_2),
$$
where $\lam d\cap\mu$ is the composition of $n$ corresponding to standard Young subgroup $d^{-1}\Sym_\lam d\cap\Sym_{\mu}$ of $\Sym_n$. In particular, $$
u^{-\ell(w)}T_w=(u^{-\ell(w_1)}T_{w_1})(u^{-\ell(d)}T_{d})(u^{-\ell(w_2)}T_{w_2}) .
$$

If $w_1\in\Sym_{\lam}$, then $f_{\tlam\tlam}T_{w_1}=u^{2\ell(w_1)}f_{\tlam\tlam}$ by\,\, \cite[Proposition 2.7]{M:gendeg}. If $w_2\in\Sym_{\lam'}$, then $T_{w_2}f_{\tllam\tllam}=(-1)^{\ell(w_2)}f_{\tllam\tllam}$ by\,\, \cite[Proposition 2.7]{M:gendeg} again.

By the above discussion, we have that $$\begin{aligned}
X_{\emptyset}&=\sum_{d\in\mathcal{D}_{\lam,\lam'}}\sum_{w\in\Sym_{\lam}d(\mathcal{D}_{\lam d\cap\lam'}\cap\Sym_{\lam'})}u^{-\ell(w)}T_w\\
&=\sum_{d\in\mathcal{D}_{\lam,\lam'}}u^{-\ell(d)}\biggl(\sum_{w_1\in\Sym_{\lam}}u^{-\ell(w_1)}T_{w_1}\biggr)T_d
\biggl(\sum_{w_2\in\mathcal{D}_{\lam d\cap\lam'}\cap\Sym_{\lam'}}u^{-\ell(w_2)}T_{w_2}\biggr)\\
\end{aligned}
$$
It follows that $$\begin{aligned}
&\quad\,f_{\tlam\tlam}X_{\emptyset}f_{\tllam\tllam}\\
&=\bigl(\sum_{w_1\in\Sym_{\lam}}u^{\ell(w_1)}\bigr)\sum_{d\in\mathcal{D}_{\lam,\lam'}}
\biggl(\sum_{w_2\in\mathcal{D}_{\lam d\cap\lam'}\cap\Sym_{\lam'}}(-u)^{-\ell(w_2)}u^{-\ell(d)}\biggr)
f_{\tlam\tlam}T_d f_{\tllam\tllam}
\end{aligned}
$$

Let $d\in\mathcal{D}_{\lam,\lam'}$. We claim that if $d^{-1}\Sym_{\lam}d\cap\Sym_{\lam'}\neq\{1\}$, then
$f_{\tlam\tlam}T_d f_{\tllam\tllam}=0$. In fact, assume that $1\neq z\in d^{-1}\Sym_{\lam}d\cap\Sym_{\lam'}\neq\{1\}$. We write $z=d^{-1}z_1d$, where $z_1\in\Sym_{\lam}$. Therefore, we get that $$\begin{aligned}
&\quad\,u^{2\ell(z_1)}f_{\tlam\tlam}T_d f_{\tllam\tllam}=f_{\tlam\tlam}T_{z_1}T_d f_{\tllam\tllam}=
f_{\tlam\tlam}T_{z_1d}f_{\tllam\tllam}=f_{\tlam\tlam}T_{dz}f_{\tllam\tllam}\\
&=(-1)^{\ell(z)}f_{\tlam\tlam}T_{d}f_{\tllam\tllam},
\end{aligned}
$$
and hence $(u^{2\ell(z_1)}-(-1)^{\ell(z)})f_{\tlam\tlam}T_{d}f_{\tllam\tllam}$=0. Since $z\neq 1$ and hence $z_1\neq 1$, it follows that $u^{2\ell(z_1)}-(-1)^{\ell(z)}\neq 0$, and hence $f_{\tlam\tlam}T_{d}f_{\tllam\tllam}$=0 as required. This proves our claim.

As a result, we can deduce that $$
\begin{aligned}
&\quad\,f_{\tlam\tlam}X_{\emptyset}f_{\tllam\tllam}\\
&=\bigl(\sum_{w_1\in\Sym_{\lam}}u^{\ell(w_1)}\bigr)\sum_{\substack{d\in\mathcal{D}_{\lam,\lam'}\\ d^{-1}\Sym_{\lam}d\cap\Sym_{\lam'}=\{1\}}}\bigl(\sum_{w_2\in\mathcal{D}_{\lam d\cap\lam'}\cap\Sym_{\lam'}}(-u)^{-\ell(w_2)}\bigr)
u^{-\ell(d)}f_{\tlam\tlam}T_d f_{\tllam\tllam}
\end{aligned}
$$

On the other hand, it is well-known that $\Sym_{\lam}w_{\lam}\Sym_{\lam'}$ is the unique double coset in $\Sym_{\lam}\!\setminus\!\Sym_n/\Sym_{\lam'}$ which has the trivial intersection property (see\,\, \cite[Proof of Lemma 4.1]{DJ1}), i.e., $w_{\lam}^{-1}\Sym_{\lam}w_{\lam}\cap\Sym_{\lam'}=\{1\}$. In particular, $\mathcal{D}_{\lam w_{\lam},\lam'}=\Sym_n$, and hence $$f_{\tlam\tlam}X_{\emptyset}f_{\tllam\tllam}
=\bigl(\sum_{w_1\in\Sym_{\lam}}u^{\ell(w_1)}\bigr)\bigl(\sum_{w_2\in\Sym_{\lam'}}(-u)^{-\ell(w_2)}\bigr)
u^{-\ell(w_{\lam})}f_{\tlam\tlam}T_{w_{\lam}}f_{\tllam\tllam} .
$$

By \cite[Proposition 2.7]{M:gendeg}and \cite[Lemma 1.5]{DJ1}, we can deduce that $$
f_{\tlam\tlam}T_{w_{\lam}}=f_{\tlam\tllam}+\sum_{w_{\lam}>z, \tlam z\in\Std(\lam)}a_z f_{\tlam,\tlam z},
$$
where $a_z\in\Q(u^2)$ for each $z$. In particular, $$
f_{\tlam\tlam}T_{w_{\lam}}f_{\tllam\tllam}=\gamma_{\tlam}f_{\tlam\tllam}\neq 0 .
$$
Note also that both $\sum_{w_1\in\Sym_{\lam}}u^{\ell(w_1)}$ and $\sum_{w_2\in\Sym_{\lam'}}(-u)^{-\ell(w_2)}$ are nonzero because they have leading terms equal to $u^{\ell(w_{\lam,0})}$ and $(-u)^{-\ell(w_{\lam',0})}$, where $w_{\lam,0}$ and $w_{\lam',0}$ are the unique longest elements in $\Sym_{\lam}$ and $\Sym_{\lam'}$ respectively. It follows that $$f_{\tlam\tlam}X_{\emptyset}f_{\tllam\tllam}\neq 0, $$ as required. This completes the proof of the lemma.
\end{proof}

\begin{cor} \label{lbcor} We have that $$
\sum_{\lam\in\mathcal{P}_n}\#\Std(\lam)=\#\{w\in\Sym_n|w^2=1\} .
$$
In particular, $$
\dim_{\Q(u)}\mathcal{H}^{\Q(u)}X_{\emptyset}\geq\#\{w\in\Sym_n|w^2=1\}.
$$
\end{cor}

\begin{proof} Let $$\begin{aligned}\pi&: \Sym_n\rightarrow\{(\s,\t)|\s,\t\in\Std(\lam),\lam\in\mathcal{P}_n\}\\
&\quad w\mapsto (P(w),Q(w))
\end{aligned}$$
be the Robinson-Schensted correspondence, cf. \cite{Ar}. By definition, $\pi$ is a bijection onto the set
$\{(\s,\t)|\s,\t\in\Std(\lam),\lam\in\mathcal{P}_n\}$, $Q(w)=P(w^{-1})$ for each $w\in\Sym_n$. It follows that $\pi$ induces a bijection between the set $I_\ast$ of the involutions in $\Sym_n$ and the set $\sqcup_{\lam\in\mathcal{P}_n}\Std(\lam)$. In particular, $\sum_{\lam\in\mathcal{P}_n}\#\Std(\lam)=\#\{w\in\Sym_n|w^2=1\}$, as required. This proves the first part of the corollary. The second part of the corollary follows from Lemma \ref{lb}.
\end{proof}

\bigskip
\section{Proof of Lusztig's Conjecture \ref{LC} when $\ast=\text{id}$ and $W=\Sym_n$}

In this section, we shall give the main result of this paper. That is, a proof of Lusztig's Conjecture \ref{LC} when $\ast=\text{id}$ and $W=\Sym_n$. Recall that $\{a_w|w^2=1,w\in\Sym_n\}$ is an $\mathcal{A}$-basis of $M$.

\begin{lem} \label{keylem1} The map $a_1\mapsto X_{\emptyset}$ can be extended to a well-defined $\Q(u)$-linear map $\eta_0$
from $\Q(u)\otimes_{\mathcal{A}}M$ to $\mathcal{H}^{\Q(u)}X_{\emptyset}$ such that for any $w\in I_{\ast}$ and any reduced $I_\ast$-expression $\sigma=(s_{j_1},\cdots,s_{j_k})$ for $w$, $$
\eta_0(a_w)=\theta_{\sigma}(X_{\emptyset}):=\theta_{\sigma,1}\circ\theta_{\sigma,2}\circ\cdots\circ\theta_{\sigma,k}(X_{\emptyset}),
$$
where for each $1\leq t\leq k$, if $$s_{j_t}(s_{j_{t+1}}\ltimes s_{j_{t+2}}\ltimes\cdots\ltimes s_{j_k})\neq (s_{j_{t+1}}\ltimes s_{j_{t+2}}\ltimes\cdots\ltimes s_{j_k})s_{j_t}>(s_{j_{t+1}}\ltimes s_{j_{t+2}}\ltimes\cdots\ltimes s_{j_k}),$$ then we define $\theta_{\sigma,t}:=T_{s_{j_t}}$; while if $$
s_{j_t}(s_{j_{t+1}}\ltimes s_{j_{t+2}}\ltimes\cdots\ltimes s_{j_k})=(s_{j_{t+1}}\ltimes s_{j_{t+2}}\ltimes\cdots\ltimes s_{j_k})s_{j_t}>(s_{j_{t+1}}\ltimes s_{j_{t+2}}\ltimes\cdots\ltimes s_{j_k}),$$
then we define $\theta_{\sigma,t}:=(T_{s_{j_t}}-u)/(u+1)$.
\end{lem}

\begin{proof} It suffices to show that the operator $\theta_{\sigma}:=\theta_{\sigma,1}\circ\theta_{\sigma,2}\circ\cdots\circ\theta_{\sigma,k}$ depends only on $w$ and not on the choice of the reduced $I_\ast$-expression $\sigma=(s_{j_1},\cdots,s_{j_k})$ for any given $w\in I_\ast$.

By Theorem \ref{mainthm0}, it suffices to show that $\theta_{\sigma}$ does not change under any one of the three basic braid $I_\ast$-transformations as introduced in Definition \ref{braid1}. So there are three possibilities:

\smallskip
{\it Case 1.} $|j_t-j_{t+1}|>1$ for some $1\leq t<k$, and the braid $I_\ast$-transformation sends $$
\sigma=(s_{j_1},\cdots,s_{j_{t-1}},s_{j_{t}},s_{j_{t+1}},s_{j_{t+2}},\cdots,s_{j_k})
$$ to $\tau:=(s_{j_1},\cdots,s_{j_{t-1}},s_{j_{t+1}},s_{j_{t}},s_{j_{t+2}},\cdots,s_{j_k})$. In this case, it follows from Corollary \ref{braid0cor2} and the fact that $T_{j_t}T_{j_{t+1}}=T_{j_{t+1}}T_{j_t}$ that $$
\theta_{\sigma,t}\circ\theta_{\sigma,{t+1}}\circ\cdots\circ\theta_{\sigma,k}(X_{\emptyset})=
\theta_{\tau,{t}}\circ \theta_{\tau,{t+1}}\circ\cdots\circ\theta_{\tau,k}(X_{\emptyset}).
$$
Hence $\theta_{\sigma}(X_{\emptyset})=\theta_{\tau}(X_{\emptyset})$ as required.

\smallskip
{\it Case 2.} $j_{t-2}=j_{t}=j_{t-1}\pm 1$ for some $3\leq t\leq k$, and the braid $I_\ast$-transformation sends $$
\sigma=(s_{j_1},\cdots,s_{j_{t-3}},s_{j_{t-2}},s_{j_{t-1}},s_{j_{t}},s_{j_{t+1}},s_{j_{t+2}},\cdots,s_{j_k})
$$ to $\tau:=(s_{j_1},\cdots,s_{j_{t-3}},s_{j_{t-1}},s_{j_{t}},s_{j_{t-1}},s_{j_{t+1}},s_{j_{t+2}},\cdots,s_{j_k})$.

Note that $t\neq k$ because otherwise $j_{k-2}=j_k=j_{k-1}\pm 1$ would imply that $s_{j_{k-2}}\ltimes s_{j_{k-1}}\ltimes s_{j_{k}}$ is not a reduced $I_\ast$-expression, a contradiction. Therefore, we must have that
$3\leq t\leq k-1$. In this case, we set $$\begin{aligned}
w:&=s_{j_{t+1}}\ltimes s_{j_{t+2}}\ltimes\cdots\ltimes s_{j_k}.\\
Z_0:&=\theta_{\sigma,{t+1}}\circ\theta_{\sigma,{t+2}}\circ\cdots\circ\theta_{\sigma,k}(X_{\emptyset}).
\end{aligned}$$
Then by Corollary \ref{braid0cor}, there are three subcases:

\smallskip
{\it Subcase 1.} $s_{j_t}w\neq ws_{j_t}$, $s_{j_{t-1}}s_{j_t}ws_{j_t}\neq s_{j_t}ws_{j_t}s_{j_{t-1}}$, $$
s_{j_t}s_{j_{t-1}}s_{j_t}ws_{j_t}s_{j_{t-1}}\neq s_{j_{t-1}}s_{j_t}ws_{j_t}s_{j_{t-1}}s_{j_t},$$
and $s_{j_{t-1}}w\neq ws_{j_{t-1}}$, $s_{j_t} s_{j_{t-1}} w s_{j_{t-1}}\neq s_{j_{t-1}} w s_{j_{t-1}} s_{j_t}$, $$
s_{j_{t-1}}s_{j_t}s_{j_{t-1}}ws_{j_{t-1}}s_{j_t}\neq s_{i_t}s_{j_{t-1}}ws_{i_{t-1}}s_{j_t}s_{j_{t-1}}.
$$

It follows from Lemma \ref{braid0} and Corollary \ref{braid0cor} that $$\begin{aligned}
&\quad\,\theta_{\sigma,{t-2}}\circ\theta_{\sigma,{t-1}}\circ\theta_{\sigma,{t}}\circ\theta_{\sigma,{t+1}}
\circ\theta_{\sigma,{t+2}}
\circ\cdots\circ\theta_{\sigma,k}
(X_{\emptyset})\\
&=T_{j_{t}}T_{j_{t-1}}T_{j_{t}}Z_0=T_{j_{t-1}}T_{j_{t}}T_{j_{t-1}}Z_0\\
&=\theta_{\tau,{t-2}}\circ \theta_{\tau,{t-1}}\circ \theta_{\tau,{t}}\circ\theta_{\tau,{t+1}}\circ\theta_{\tau,{t+2}}\circ\cdots\circ
\theta_{\tau,k}(X_{\emptyset}).
\end{aligned}
$$
Hence $\theta_{\sigma}(X_{\emptyset})=\theta_{\tau}(X_{\emptyset})$ as required.

\smallskip
{\it Subcase 2.} $s_{j_t}w\neq ws_{j_t}$, $s_{j_{t-1}}s_{j_t}ws_{j_t}\neq s_{j_t}ws_{j_t}s_{j_{t-1}}$, $$
s_{j_t}s_{j_{t-1}}s_{j_t}ws_{j_t}s_{j_{t-1}}= s_{j_{t-1}}s_{j_t}ws_{j_t}s_{j_{t-1}}s_{j_t},$$
and $s_{j_{t-1}}w=ws_{j_{t-1}}$, $s_{j_t}s_{j_{t-1}}w\neq s_{j_{t-1}}ws_{j_t}$, $$
s_{j_{t-1}}s_{j_t}s_{j_{t-1}}ws_{j_t}\neq s_{j_t}s_{j_{t-1}}ws_{j_t}s_{j_{t-1}}.$$

It follows from Lemma \ref{braid0} and Corollary \ref{braid0cor} that $$\begin{aligned}
&\quad\,\theta_{\sigma,t-2}\circ\theta_{\sigma,t-1}\circ\theta_{\sigma,t}\circ\theta_{\sigma,t+1}
\circ\theta_{\sigma,t+2}
\circ\cdots\circ\theta_{\sigma,k}
(X_{\emptyset})\\
&=\frac{T_{j_{t}}-u}{u+1}T_{j_{t-1}}T_{j_{t}}Z_0=T_{j_{t-1}}T_{j_{t}}\frac{T_{j_{t-1}}-u}{u+1}Z_0\\
&=\theta_{\tau,{t-2}}\circ \theta_{\tau,{t-1}}\circ \theta_{\tau,{t}}\circ\theta_{\tau,{t+1}}\circ\theta_{\tau,{t+2}}\circ\cdots\circ\theta_{\tau,k}(X_{\emptyset}).
\end{aligned}
$$
Hence $\theta_{\sigma}(X_{\emptyset})=\theta_{\tau}(X_{\emptyset})$ as required.

\smallskip
{\it Subcase 3.} $s_{j_t}w=ws_{j_t}$, $s_{j_{t-1}}s_{j_t}w\neq s_{j_t}ws_{j_{t-1}}$, $$
s_{j_t}s_{j_{t-1}}s_{j_t}ws_{j_{t-1}}\neq s_{j_{t-1}}s_{j_t}ws_{j_{t-1}}s_{j_t},$$
and $s_{j_{t-1}}w\neq ws_{j_{t-1}}$, $s_{j_t}s_{j_{t-1}}ws_{j_{t-1}}\neq s_{j_{t-1}}ws_{j_{t-1}}s_{j_t}$, $$
s_{j_{t-1}}s_{j_t}s_{j_{t-1}}ws_{j_{t-1}}s_{j_t}=s_{j_t}s_{j_{t-1}}ws_{j_{t-1}}s_{j_t}s_{j_{t-1}} .
$$

It follows from Lemma \ref{braid0} and Corollary \ref{braid0cor} that $$\begin{aligned}
&\quad\,\theta_{\sigma,t-2}\circ\theta_{\sigma,t-1}\circ\theta_{\sigma,t}\circ\theta_{\sigma,t+1}
\circ\theta_{\sigma,t+2}
\circ\cdots\circ\theta_{\sigma,k}
(X_{\emptyset})\\
&=T_{j_{t}}T_{j_{t-1}}\frac{T_{j_{t}}-u}{u+1}Z_0=\frac{T_{j_{t-1}}-u}{u+1}T_{j_{t}}T_{j_{t-1}}Z_0\\
&=\theta_{\tau,{t-2}}\circ \theta_{\tau,{t-1}}\circ \theta_{\tau,{t}}\circ\theta_{\tau,{t+1}}\circ\theta_{\tau,{t+2}}\circ\cdots\circ\theta_{\tau,k}(X_{\emptyset}).
\end{aligned}
$$
Hence $\theta_{\sigma}(X_{\emptyset})=\theta_{\tau}(X_{\emptyset})$ as required.

\smallskip
{\it Case 3.} $|j_{k-1}-j_k|=1$, and the braid $I_\ast$-transformation sends $$
\sigma=(s_{j_1},\cdots,s_{j_{k-3}},s_{j_{k-2}},s_{j_{k-1}},s_{j_{k}})
$$ to $\tau:=(s_{j_1},\cdots,s_{j_{k-3}},s_{j_{k-2}},s_{j_{k}},s_{j_{k-1}})$. Without loss of generality, we can assume that $j_{k-1}=j_k+1$. For simplicity, we set $a:=j_k$, then $j_{k-1}=a+1$.

Let $\mathcal{D}(a)$ be the set of minimal length distinguished right coset representatives of $\Sym_{\{a,a+1\}}$ in $\Sym_n$. Then it is clear that \begin{equation}\label{2terms}
X_{\emptyset}=\biggl(\sum_{w\in\Sym_{\{a,a+1\}}}u^{-\ell(w)}T_w\biggr)\biggl(\sum_{z\in\mathcal{D}(a)}u^{-\ell(z)}T_z
\biggr).
\end{equation}
In this case, all we want to do is to show that $$
T_{s_a}\frac{T_{s_{a+1}}-u}{u+1}X_{\emptyset}=T_{s_{a+1}}\frac{T_{s_a}-u}{u+1}X_{\emptyset}.$$
By (\ref{2terms}), it suffices to show that $$
T_{s_a}\frac{T_{s_{a+1}}-u}{u+1}\sum_{w\in\Sym_{\{a,a+1\}}}u^{-\ell(w)}T_w=
T_{s_{a+1}}\frac{T_{s_a}-u}{u+1}\sum_{w\in\Sym_{\{a,a+1\}}}u^{-\ell(w)}T_w .
$$
By direct verification, we can get that $$\begin{aligned}
&\quad\,T_{s_a}\frac{T_{s_{a+1}}-u}{u+1}\sum_{w\in\Sym_{\{a,a+1\}}}u^{-\ell(w)}T_w=T_{s_a}\frac{T_{s_{a+1}}-u}{u+1}
\biggl(1+u^{-1}T_{s_a}+u^{-1}T_{s_{a+1}}\\
&\qquad +u^{-2}T_{s_a}T_{s_{a+1}}+u^{-2}T_{s_{a+1}}T_{s_a}+u^{-3}T_{s_a}T_{s_{a+1}}T_{s_a}
\biggr)\\
&=(u-u^{-1})(T_{s_a}T_{s_{a+1}}+T_{s_{a+1}}T_{s_a})+(1+u+u^{-3}-u^{-2}-2u^{-1})T_{s_a}T_{s_{a+1}}T_{s_a}\\
&=T_{s_{a+1}}\frac{T_{s_{a}}-u}{u+1}\biggl(1+u^{-1}T_{s_a}+u^{-1}T_{s_{a+1}}+u^{-2}T_{s_a}T_{s_{a+1}}
+u^{-2}T_{s_{a+1}}T_{s_a}+u^{-3}T_{s_a}T_{s_{a+1}}T_{s_a}\biggr)\\
&=T_{s_{a+1}}\frac{T_{s_{a}}-u}{u+1}\sum_{w\in\Sym_{\{a,a+1\}}}u^{-\ell(w)}T_w ,
\end{aligned}
$$
as required. This completes the proof of the lemma.
\end{proof}

\begin{lem}\label{keylem2} With the notations as in Lemma \ref{keylem1}, the $\Q(u)$-linear map $\eta_0$ is a
left $\mathcal{H}^{\Q(u)}$-module homomorphism. In particular, $\eta_0=\eta$ is a well-defined surjective left $\mathcal{H}^{\Q(u)}$-homomorphism from $\Q(u)\otimes_{\mathcal{A}}M$ onto $\mathcal{H}^{\Q(u)}X_{\emptyset}$.
\end{lem}

\begin{proof} Once we can prove that $\eta_0$ is a left $\mathcal{H}^{\Q(u)}$-module homomorphism, then it follows immediately that $\eta$ is well-defined and $\eta_0=\eta$ because both of them send $a_1$ to $X_{\emptyset}$.

Since $\{a_w|w\in I_\ast\}$ is an $\mathcal{A}$-basis of $M$ (and hence a $\Q(u)$-basis of $\Q(u)\otimes_{\mathcal{A}}M$), in order to show that $\eta_0$ is a left $\mathcal{H}^{\Q(u)}$-module homomorphism, it suffices to show that for any
$w\in I_\ast$ and any $1\leq k<n$, \begin{equation}\label{keyeq}
\eta_0(T_{s_k} a_w)=T_{s_k}\eta_0(a_w) .
\end{equation}

We use induction on $\rho(w)$ to prove (\ref{keyeq}). If $\rho(w)=0$ then $w=1$. In this case, by the definition of $\eta_0$ in Lemma \ref{keylem1}, $$
\eta_0(T_{s_k} a_w)=\eta_0(T_{s_k} a_1)=\eta_0(a_{s_k})=T_{s_k}X_{\emptyset}=T_{s_k}\eta_0(a_1),
$$
as required.
In general, let $m\in\N$. Suppose that for any $1\leq k<n$ and any $w'\in I_\ast$ with $\rho(w')<m$, we have that $$ \eta_0(T_{s_k} a_{w'})=T_{s_k}\eta_0(a_{w'}) .
$$
Now let $w\in I_\ast$ with $\rho(w)=m$. There are two possibilities:

\smallskip
{\it Case 1.} $s_kw>w$. If $s_kw\neq ws_k$, then by the definition of $\eta_0$ in Lemma \ref{keylem1}, $$
\eta_0(T_{s_k} a_w)=\eta_0(a_{s_k\ltimes w})=T_{s_k}\eta_0(a_w),
$$
as required. If $s_kw=ws_k$, then by the definition of $\eta_0$ in Lemma \ref{keylem1}, $$
\eta_0(\frac{T_{s_k}-u}{u+1}a_w)=\eta_0(a_{s_k\ltimes w})=\frac{T_{s_k}-u}{u+1}\eta_0(a_w).
$$
It follows that $\eta_0(T_{s_k} a_w)=T_{s_k}\eta_0(a_w)$ still holds in this case.

\smallskip
{\it Case 2.} $s_kw<w$. By Corollary \ref{deletion2}, we can write $w=s_k\ltimes w'$ with $\rho(w)=\rho(w')+1$.
In particular, $\rho(w')=\rho(w)-1=m-1<m$. If $s_kw'\neq w's_k$, then by induction hypothesis and Theorem \ref{LVaction}, we have that $$
\eta_0(a_w)=\eta_0(T_{s_k}a_{w'})=T_{s_k}\eta_0(a_{w'}).
$$
It follows from induction hypothesis and Lemma \ref{keylem1} that $$\begin{aligned}
\eta_0(T_{s_k}a_w)&=\eta_0((T_{s_k})^2a_{w'})=\eta_0\bigl((u^2-1)T_{s_k}a_{w'}+u^2a_{w'}\bigr)\\
&=(u^2-1)\eta_0(T_{s_k}a_{w'})+u^2\eta_0(a_{w'})=(u^2-1)T_{s_k}\eta_0(a_{w'})+u^2\eta_0(a_{w'})\\
&=T_{s_k}T_{s_k}\eta_0(a_{w'})=T_{s_k}\eta_0(T_{s_k}a_{w'})=T_{s_k}\eta_0(a_w),
\end{aligned}
$$
as required.

If $s_kw'=w's_k$, then $$
\eta_0(a_w)=\eta_0(\frac{T_{s_k}-u}{u+1}a_{w'})=\frac{T_{s_k}-u}{u+1}\eta_0(a_{w'}).
$$
It follows from induction hypothesis and Lemma \ref{keylem1} that $$\begin{aligned}
\eta_0(T_{s_k}a_w)&=\eta_0(T_{s_k}\frac{T_{s_k}-u}{u+1}a_{w'})=\eta_0\bigl(\frac{(u^2-u-1)T_{s_k}}{u+1}a_{w'}
+\frac{u^2}{u+1}a_{w'}\bigr)\\
&=\frac{(u^2-u-1)}{u+1}\eta_0(T_{s_k}a_{w'})+\frac{u^2}{u+1}\eta_0(a_{w'})\\
&=\frac{(u^2-u-1)}{u+1}T_{s_k}\eta_0(a_{w'})+\frac{u^2}{u+1}\eta_0(a_{w'})\\
&=T_{s_k}\frac{T_{s_k}-u}{u+1}\eta_0(a_{w'})=T_{s_k}\eta_0(\frac{T_{s_k}-u}{u+1}a_{w'})=T_{s_k}\eta_0(a_w),
\end{aligned}
$$
as required.
\end{proof}

\begin{thm}\label{mainthm2} Lusztig's conjecture (\ref{LC}) is true in the case when $\ast=\text{id}$ and $W=\Sym_n$.
\end{thm}

\begin{proof} By Lemma \ref{keylem2}, $\eta$ defines a surjective left $\mathcal{H}^{\Q(u)}$-module homomorphism from $\Q(u)\otimes_{\mathcal{A}}M$ onto $\mathcal{H}^{\Q(u)}X_{\emptyset}$. On the other hand, by Lemma
\ref{keylem1}, $$
\dim_{\Q(u)}\mathcal{H}^{\Q(u)}X_{\emptyset}\geq\#\{w\in\Sym_n|w^2=1\}=\dim_{\Q(u)}\Q(u)\otimes_{\mathcal{A}}M .
$$
It follows that $\eta$ must be a left $\mathcal{H}^{\Q(u)}$-module isomorphism.
\end{proof}


\bigskip

\begin{thebibliography}{3}


\bibitem{Ar}
{\sc S.~Ariki}, {\em Robinson-Schensted correspondence and left cells}, in: Combinatorial Methods in Representation Theory, Kyoto, 1998, in: Adv. Stud. Pure Math., {\bf 28} 2000, 1--20.

\bibitem{DJ1}
{\sc R.~Dipper and G.~James}, {\em Representations of Hecke algberas of general linear groups},  Proc.
London Math. Soc. {\bf 52} (1986), 20--52.

\bibitem{Hu1}
{\sc A.~Hultman}, {\em Fixed points of involutive automorphisms of the Bruhat order}, Adv. Math., {\bf 195}(1) (2005), 283-¨C296.

\bibitem{Hu2}
\leavevmode\vrule height 2pt depth -1.6pt width 23pt,
{\em The Combinatorics of twisted involutions in Coxeter groups}, Trans. Amer. Math. Soc., {\bf 359}(6) (2007), 2787--2798.

\bibitem{KL}
{\sc D.~Kazhdan, G.~Lusztig}, {\em Representations of Coxeter groups and Hecke algebras}, Invent. Math., {\bf 53} (1979), 165--184.

\bibitem{Lu1}
{\sc G.~Lusztig}, {\em A bar operator for involutions in a Coxeter groups}, (2012), preprint, arXiv:1112.0969.

\bibitem{Lu2}
{\sc G.~Lusztig}, {\em Asymptotic Hecke algebras and involutions}, (2012), preprint, arXiv:1204.0276.

\bibitem{LV1}
{\sc G.~Lusztig, D.~Vogan}, {\em Hecke algebras and involutions in Weyl groups}, Bulletin of the Institute of Mathematics Academia Sinica (New eries), {\bf 7}(3) (2012), 323--354.

\bibitem{Mar}
{\sc E.~Marberg}, {\em Positivity conjectures for Kazhdan--Lusztig theory on twisted involutions: the universal case}, Representation theory, {\bf 18} (2014), 88--116.

\bibitem{Mathas} {\sc A.~Mathas}, {\em Iwahori-Hecke algebras and Schur algebras of the symmetric group},
  University Lecture Series {\bf 15}, American Mathematical Society, 1999.

\bibitem{M:gendeg}
\leavevmode\vrule height 2pt depth -1.6pt width 23pt,
 {\em Matrix units and generic degrees for the {A}riki-{K}oike algebras}, J. Algebra, {\bf 281}
  (2004), 695--730.

\bibitem{Vor}
{\sc K.~Vorwerk}, {\em The Bruhat order on involutions and pattern avoidance}, Diplomarbeit, Technische Universit\"at Chemntiz, 2007.

\end{thebibliography}
\end{document}